%% file: TOKU_MFUKU_SISDP.tex
\RequirePackage{fix-cm}
\documentclass[smallextended]{svjour3}       
\smartqed  
\usepackage{graphicx}
\usepackage{color}
\usepackage{amsmath}
\usepackage{latexsym}
\begin{document}
\newcommand{\mukj}{\mu_{k_j-1}}
\newcommand{\skhj}{s^{k_j}_{\fr}}
\newcommand{\fr}{\frac{1}{2}}
\newcommand{\vexact}{v^{\rm exact}}
\newcommand{\vout}{v^{\rm out}}
\newcommand{\var}{\sharp\,\mathrm{var}}
\newcommand{\bV}{\overline{V}}
\newcommand{\by}{\bar{y}}
\newcommand{\bz}{\bar{z}}
\newcommand{\K}{\mathcal{K}}            
\newcommand{\R}{\mathcal{R}}
\newcommand{\y}{y^{(i)}}
\newcommand{\square}{\Box}
\newcommand{\skh}{s^k_{\frac{1}{2}}}
\newcommand{\sk}{s^k_1}
\renewcommand{\l}{\ell}
\newcommand{\hw}{\hat{w}}
\newcommand{\hx}{\hat{x}}
\newcommand{\hy}{\hat{y}}
\newcommand{\hV}{\hat{V}}
\newcommand{\hP}{\hat{P}}
\newcommand{\hmu}{\hat{\mu}}
\newcommand{\hzeta}{\hat{\zeta}}
\newcommand{\dDelta}{\Delta}
\newcommand{\bmu}{\bar{\mu}}
\newcommand{\oDelta}{\Delta_{1}}
\newcommand{\jDelta}{\Delta_{j}}
\newcommand{\fDelta}{\dDelta_{\fr}}
\newcommand{\pkh}{P_{k+\frac{1}{2}}}
\newcommand{\wlim}{{\rm w}^\ast\mbox{-}\lim}
\newcommand{\wast}{\mbox{weak}^\ast}
\newcommand{\wlyast}{\mbox{weakly}^\ast}
\renewcommand{\epsilon}{\varepsilon}
\renewcommand{\t}{t}
\newcommand{\tbkpre}{\mu_{k-1}}
\newcommand{\ykdt}{dy^k(\tau)}
\newcommand{\yadt}{dy^{\ast}(\tau)}
\newcommand{\ydt}{dy(\tau)}
\newcommand{\ydtp}{dy_{+}(\tau)}
\newcommand{\supp}{{\rm supp}}
\newcommand{\tblumuk}{\mu_k}
\renewcommand{\phi}{\varphi}
\newcommand{\hxi}{\hat{\xi}}
\newcommand{\tact}{T_{act}(\Bar{x})}
\newcommand{\bx}{\Bar{x}}
\newcommand{\tepsk}{\tilde{\epsilon}_{k}}
\newcommand{\F}{\mathcal{F}}
\newcommand{\M}{\mathcal{M}}
\newcommand{\sastk}{s^k_{0}}
\newcommand{\temp}
{\left(1+2\theta\mu_{k-1}+\sqrt{1+4\theta\left(\mu_{k-1}-\theta\mu_{k-1}^{1+\alpha}\right)}\right)}
\newcommand{\bt}{\Bar{t}}
\newcommand{\co}{{\mathop{\mathrm{co}\,}}}
\newcommand{\bd}{{\mathop{\mathrm{bd}\,}}}
\newcommand{\cl}{{\mathop{\mathrm{cl}\,}}}
\newcommand{\cone}{{\mathop{\mathrm{cone}}}}
\newcommand{\inn}{{\mathop{\mathrm{int}\,}}}
\newcommand{\argmin}{{\mathop{\mathrm{argmin}}}}
\newcommand{\Min}{{\mathop{\mathrm{Minimize}}}}
\newcommand{\dist}{{\mathop{\mathrm{dist}\,}}}
\renewcommand{\refname}{References}
\newcommand{\Real}{\mathop{\mathrm{Re}}\,}
\newcommand{\Image}{\mathop{\mathrm{Im}}\,}
\newcommand{\PV}{\left(P\odot P\right)V}
\newcommand{\PF}{\left(P^{-\top}\odot P^{-\top}\right)F(x)}
\newcommand{\reps}{r_{\epsilon}}
\newcommand{\FiV}{{\left(F_i\bullet V\right)_{i=1}^n}}
\newcommand{\FiWk}{{\left(F_i\bullet W_k\right)_{i=1}^n}}
\newcommand{\FiVk}{{\left(F_i\bullet V_k\right)_{i=1}^n}}
\newcommand{\FiWast}{{\left(F_i\bullet W_{\ast}\right)_{i=1}^n}}
\newcommand{\FiVast}{{\left(F_i\bullet V_{\ast}\right)_{i=1}^n}}
\newcommand{\svecFi}{{\left({\rm svec}(F_i) \right)_{i=1}^n}}
 \newcommand{\iteout}{{\rm ite}_{\rm out}}
\def\ru[#1][#2]{{#1}^{#2}}
\def\rl[#1][#2]{{#1}_{#2}}
\newcommand{\ix}{x_{\rm out}}
\newcommand{\iU}{U_{\rm out}}
\newcommand{\iV}{V_{\rm out}}
 \newcommand{\iyeq}{z_{\rm out}}
\newcommand{\iy}{y_{i,{\rm out}}}
\newcommand{\tsocp}{t_{\rm socp}}
\newcommand{\tadd}{t_{\rm add}}
\newcommand{\iteopt}{{\rm ite}_{\rm opt}}
\newcommand{\tred}{\textcolor{red}}
\newcommand{\tblue}{\textcolor{blue}}
\newcommand{\trFV}{\begin{pmatrix}F_1\bullet V\\ \vdots\\ F_n\bullet V\end{pmatrix}}
\newcommand{\trFW}{\begin{pmatrix}F_1\bullet V\\\vdots\\ F_n\bullet V\end{pmatrix}}
\newcommand{\trFWast}{\begin{pmatrix}DF_1\bullet W_{\ast}\\ \vdots\\ DF_n\bullet W_{\ast}\end{pmatrix}}
\newcommand{\A}{\mathcal{A}}
\newcommand{\B}{\mathcal{B}}
\newcommand{\C}{\mathcal{C}}
\newcommand{\dx}{\Delta x}
\newcommand{\dy}{\Delta y}
\newcommand{\dz}{\Delta z}
\newcommand{\dV}{\Delta V}
\newcommand{\dF}{\Delta F}
\newcommand{\dw}{\Delta w}
\newcommand{\dhw}{\Delta\hat{W}}
\newcommand{\dhx}{\Delta\hat{x}}
\newcommand{\dhy}{\Delta\hat{y}}
\newcommand{\dhz}{\Delta\hat{z}}
\newcommand{\dhV}{\Delta\hat{V}}
\newcommand{\svec}{\mathop{\rm svec}}
\newcommand{\Phifb}{\Phi_{\rm FB}}
\newcommand{\Rfb}{R_{\rm FB}}
\newcommand{\yt}[1]{y_{\tau_{#1}}}
\newcommand{\DF}[2]{DF_{x_{#2}}}
\newcommand{\bxopt}{x^{\ast}}
\newcommand{\Wsol}{W_{\rm sol}}
\newcommand{\buxast}{\bar{U}{(x^{\ast})}}
\newcommand{\atau}{\tau^{\ast}}
\newcommand{\pa}{p_{\ast}}
\newcommand{\mpw}{\mathcal{W}^{\prime}}
\newcommand{\mukds}{\mu^{\delta_2}_k}
\makeatletter
\@addtoreset{equation}{section}
\def\theequation{\thesection.\arabic{equation}}
\makeatother
\newtheorem{lemmaA}{Lemma~A}{\bf}{\it}
\title{
Primal-dual path following method for 
nonlinear semi-infinite programs with semi-definite constraints
\thanks{
The work was supported by JSPS KAKENHI Grant Number [15K15943].
}
}


\author{Takayuki Okuno\and
        Masao Fukushima
}


\institute{T. Okuno \at
RIKEN, The Center for Advanced Intelligence Project (AIP),
Nihonbashi 1-chome Mitsui Building, 15th floor,1-4-1 Nihonbashi, Chuo-ku, Tokyo 103-0027, Japan              
 \\              \email{takayuki.okuno.ks@riken.jp}           
           \and
           M. Fukushima \at
              Nanzan University, Faculty of Science and Engineering, 
18 Yamazato-cho, Showa-ku, Nagoya 466-8673, Japan
\email{fuku@nanzan-u.ac.jp}
}

\date{Received: date / Accepted: date}

\maketitle

\begin{abstract}
In this paper, we propose {two} algorithms for nonlinear semi-infinite semi-definite programs with infinitely many convex inequality constraints, called SISDP for short. 
A straightforward approach to the SISDP is to use classical methods for semi-infinite programs such as discretization and exchange methods and solve a sequence of (nonlinear) semi-definite programs (SDPs).
However, it is often too demanding to find exact solutions of SDPs.  

Our first approach does not rely on solving SDPs but on approximately following {a path leading to a solution}, which is formed on the intersection of the semi-infinite {region} and the interior of the semi-definite {region}.
We show weak* convergence of this method to a Karush-Kuhn-Tucker point 
of the SISDP
under some mild assumptions and further provide with sufficient conditions for strong convergence. 
Moreover, 
as the second method, to achieve fast local convergence, we 
integrate a two-step sequential quadratic programming method 
{equipped} with Monteiro-Zhang scaling technique into the first method. 
We particularly prove two-step superlinear convergence of the second
method using Alizadeh-Hareberly-Overton-like, Nesterov-Todd, and Helmberg-Rendle-Vanderbei-Wolkowicz/Kojima-Shindoh-Hara/Monteiro scaling directions. 
Finally, we conduct some numerical experiments to demonstrate the efficiency of the proposed method through comparison with a discretization method that solves SDPs
obtained by finite relaxation of the SISDP.
\end{abstract}

\section{Introduction}
In this paper, we consider the following nonlinear semi-infinite 
semi-definite program with an infinite number of convex inequality constraints and one linear matrix inequality constraint, SISDP for short:
{\begin{align}
 \begin{array}{ll}
  \displaystyle{\mathop{\rm Minimize}}  &f(x)
  \vspace{0.5em}\\
  {\rm subject~to} &g(x,\tau)\le 0\ \mbox{ for all } {\tau}\in T, \\
                   &{F}(x)\in {S^m_{+}},\\
 \end{array}\label{lsisdp}
\end{align}
where 
$f:\R^n\to \R$ is a continuously differentiable function and $T$ is a compact metric space.
In addition, $g:\R^n\times T\to \R$ is a continuous function, and $g(\cdot,\tau)$ is supposed to be convex and continuously differentiable. 
Moreover, $S^m$ and $S^m_{++} (S^m_{+})$ denote the sets of 
$m\times m$ symmetric matrices 
and symmetric positive (semi-)definite matrices, respectively, and 
$F(\cdot):\R^n\to S^{m}$ is an affine function, i.e.,
$$
F(x):=F_0+\sum_{i=1}^nx_iF_i
$$
with $F_i\in S^m$ for $i=0,1,\ldots,n$ and $x=(x_1,x_2,\ldots,x_n)^{\top}$.
We assume that the SISDP\,\eqref{lsisdp} has a nonempty solution set.
We may let the SISDP\,\eqref{lsisdp} include linear equality constraints, to which the algorithms and theories 
given in the subsequent sections can be extended straightforwardly. But, for simplicity of expression, we omit them.

When $T$ comprises a finite number of elements, the SISDP reduces to a nonlinear semi-definite program (nonlinear SDP or  NSDP).
Particularly when all the functions are affine with respect to $x$, it further reduces to the linear SDP (LSDP).
As is known broadly, studies on the LSDP have been crucially promoted in the aspects of theory, algorithms, and applications\,\cite{wolkowicz2012handbook}.
Compared with the LSDP, studies on the NSDP are still scarce, although
important applications are found in various areas\,\cite{freund2007nonlinear,konno2003cutting,leibfritz2009successive}.
Shapiro\,\cite{shapiro1997first} expanded an elaborate theory on the first and second order optimality conditions of the NSDP.
See \cite{BonSp} for a comprehensive description of the optimality conditions and duality theory of the NSDP.
Yamashita et al.\,\cite{yabe} proposed a primal-dual interior point-type method using the Monteiro-Zhang (MZ) directions family and showed its global convergence property. They further made local convergence analysis in \cite{yamashita2012local}.
The SQP method {for nonliear programs} was also {extended} to the NSDP by Freund et al.\,\cite{freund2007nonlinear}.
See the survey article\,\cite{yamashita2015survey} for {more} algorithms
designed to solve the NSDP. 

In the absence of the semi-definite constraint, \eqref{lsisdp} becomes a nonlinear semi-infinite program (SIP) with an infinite number of convex constraints. 
For solving nonlinear SIPs, many researchers proposed various kinds of algorithms, for example
discretization based methods\,\cite{reemtsen1991discretization,still2001discretization}, local reduction based methods\,\cite{gramlich1995local,pereira2011interior,pereira2009reduction,Tanaka}, Newton-type methods\,\cite{li2004smoothing,qi2003semismooth}, smoothing projection methods\,\cite{xu2014solving}, convexification based methods\,\cite{floudas2007adaptive,shiu2012relaxed,stein2012adaptive,wang2015feasible}, and so on.  
For an overview of the SIP, see \cite{sip-recent,sip2,Reem} and the references therein.

Most closely related to the SISDP\,\eqref{lsisdp} are SIPs involving (possibly infinitely many) conic constraints. 
Li et al.\,\cite{li2004solution} considered a linear SIP with semi-definite constraints and proposed a discretization based method.
Subsequently, Li et al.\,\cite{li2006relaxed} tackled the same problem and developed a relaxed cutting plane method.                                                              
Hayashi and Wu \cite{hayashi4} focused on a linear SIP involving second-order cone (SOC) constraints and proposed an exchange-type method. It is worth mentioning that the SISDP\,\eqref{lsisdp} can be viewed as a generalization of those problems. 
More recently, Okuno et al.\,\cite{okuno2012regularized} considered a convex SIP with an infinite number of conic constraints, and proposed an exchange-type method combined with Tikhonov's regularization technique. Okuno and Fukushima\,\cite{okuno2014local} restricted themselves to a nonlinear SIP with infinitely many SOC constraints, and constructed a quadratically convergent sequential quadratic programming (SQP)-type method based on the local reduction method. One of common features of the algorithms mentioned above is to solve a sequence of certain conic constrained problems.

We can find some important applications of the SISDP.
For example, 
semi-infinite eigenvalue optimization problems\,\cite{li2004solution}, 
finite impulse response (FIR) filter design problems\,\cite{spwu1996}, and
robust envelop-constrained filter design with orthonormal bases\,\cite{li2007robust} can be formulated as the SISDP whose functions are all affine with respect to $x$.
{Moreover,} robust beam forming problems\,\cite{yu2008novel} can be formulated as the SISDP with infinitely many nonlinear {inequality constraints}.
However, to the best of our knowledge, there is no existing work that deals with the SISDP\,\eqref{lsisdp} itself.

{In this paper, we propose two algorithms tailored to the SISDP.
In the first method, we generate a sequence approaching a Karush-Kuhn-Tucker (KKT) point of the SISDP
by approximately following a central path formed by barrier KKT (BKKT) points of the SISDP. 
The BKKT points, whose definition will be provided in Section\,\ref{sec:3}, can be computed efficiently using the interior-point SQP-type method proposed in the authors' recent work\,\cite{okuno2018sc}. 
Although it is possible to design a convergent algorithm that solves NSDPs iteratively like the existing algorithms mentioned in the previous paragraph, it is often too demanding to get an accurate solution of an NSDP at each iteration.
In contrast, the proposed path-following algorithm will only require solving quadratic programs if it is combined with the interior point SQP-type method. 
In the second method, to accelerate the local convergence speed, we {further}
integrate a two-step {SQP} method into the first method.
{Specifically, we derive the scaled barrier KKT system of the SISDP by means of the local reduction method\,\cite{gramlich1995local,pereira2011interior,pereira2009reduction,Tanaka}
and the Monteiro-Zhang scaling technique\,\cite[Chapter~10]{wolkowicz2012handbook}.
We then perform a two-step SQP method to generate iteration points, while decreasing a barrier parameter to zero superlinearly.
In each step of the two-step SQP, 
to produce a search direction, 
we solve
a mixed linear complementarity system
approximating the aforementioned scaled barrier KKT system, which can be solved via a certain quadratic program.
We then adjust a step-size along the obtained search direction so that the next iteration point remains to lie in the interior of the semi-definite region.
We will show that, under some regularity conditions at a KKT point of the SISDP, a step-size of the unity is eventually adopted and two-step superlinear convergence is achieved. 

The proposed methods may be viewed as an extension of the primal-dual interior point method\,\cite{yabe} for the NSDPs. 
Nonetheless, the theoretical and algorithmic extensions are not straightforward because of the {presence} of infinitely many inequality constraints.
Furthermore, the results {obtained in the paper} have novelty not only in the field of the SIP but also the NSDP.

The paper is organized as follows: 
In Section~\ref{sec:3}, we propose a primal-dual path-following method for the SISDP.
{We prove that {any} $\wast$-accumulation point of the generated sequence is a KKT point of the SISDP under some mild assumptions. We also give a sufficient condition for strong convergence of the sequence.} 
In Section~\ref{sec:4}, we further combine the local-reduction based SQP method with the prototype method and prove that it converges to a KKT point of the SISDP two-step superlinearly.
In Section~\ref{sec:5}, we conduct some numerical experiments to exhibit the efficiency of the proposed method.
Finally, we conclude this paper with some remarks.
\subsection*{Notations}
Throughout this paper, we use the following notations: 
The identity matrix is denoted by $I$.
For any $P\in \R^{m\times m}$, ${\rm Tr}(P)$ denotes the trace of $P$.
For any symmetric matrices $X,Y\in S^m$, we denote the Jordan product of $X$ and $Y$ by $X\circ Y:=(XY+YX)/2$ and the inner product of $X$ and $Y$ by 
$X\bullet Y={\rm Tr}(XY)$. 
Also, we denote the Frobenius norm $\|X\|_F:=\sqrt{X\bullet X}$ 
and 
\begin{align*}
{\rm svec}(X):=&
(X_{11},\sqrt{2}X_{21},\ldots,\sqrt{2}X_{m1},X_{22},\\
&\hspace{3em}\sqrt{2}X_{32},\ldots,\sqrt{2}X_{m2},X_{33},\ldots,X_{mm})^{\top}\in \R^{\frac{m(m+1)}{2}}
\end{align*}
for $X\in S^m$.
We write $\FiV:=\left(F_1\bullet V,F_2\bullet V,\ldots,F_n\bullet V\right)^{\top}\in \R^n$ 
for $V,F_1,F_2,\ldots,F_n\in S^m$.
For any $X\in S^m$, we define the linear operator $\mathcal{L}_X:S^m\to S^m$ by $\mathcal{L}_X(Z):=X\circ Z$.
We also denote $({\zeta})_+:=\max({\zeta},0)$ for any $\zeta\in \R$.
For sequences $\{y^k\}$ and $\{z^k\}$, if $\|y^k\|\le M\|z^k\|$
for any $k$ with some $M>0$,
we write $\|y^k\|=O(\|z^k\|)$. If $M_1\|z^k\|\le \|y^k\|\le M_2\|z^k\|$
for any $k$ with some $M_1,M_2>0$,  we represent $\|y^k\|=\Theta(\|z^k\|)$.
Moreover, if there exists a sequence $\{\alpha_k\}$ with $\lim_{k\to\infty}\alpha_k=0$ and $\|y^k\|\le \alpha_k\|z^k\|$ for any $k$,
we write $\|y^k\|=o(\|z^k\|)$. 
Finally, we let $\perp$ denote the perpendicularity.
\subsection*{Terminologies from functional analysis}
Let us review some terminologies from functional analysis briefly.
For more details, refer to the basic material \cite[Section~2]{BonSp} or suitable textbooks of functional analysis.

Let $\mathcal{C}(T)$ be the set of real-valued continuous functions defined on $T$ endowed with the supremum norm $\|h\|:=\max_{\tau\in T}|h(\tau)|$. 
Let $\mathcal{M}(T)$ be the dual space of $\mathcal{C}(T)$, which can be identified with the space of (finite signed) regular Borel measures with the Borel sigma algebra $\mathcal{B}$ on $T$ equipped with the total variation norm, i.e.,
$\|y\|:=\sup_{A\in \mathcal{B}}y(A)-\inf_{A\in \mathcal{B}}y(A)$ for $y\in \mathcal{M}(T)$. 
Denote by $\mathcal{M}_+(T)$ the set of all the nonnegative Borel measures of $\mathcal{M}(T)$.  
Especially if $y\in \mathcal{M}_+(T)$, $\|y\|=y(T)$ since $\inf_{A\in \mathcal{B}}y(A)=y(\emptyset)=0$
and $\sup_{A\in \mathcal{B}}y(A)=y(T)$.
We say that $y\in \M(T)$ is a finite discrete measure if there exist a finite number of indices $\tau_1,\tau_2,\ldots,\tau_q\in T$ and scalars $\alpha_1,\alpha_2,\ldots,\alpha_q\in \R$ such that $y(A)=\sum_{i=1}^q\alpha_i\delta_A(\tau_i)$  
for any Borel set $A\in \mathcal{B}$, where $\delta_S:T\to \R$ is the indicator function satisfying 
$\delta_S(\tau)=1$ if $\tau\in S$ and $\delta_S(\tau)=0$ otherwise.

Let 
$\langle\cdot,\cdot\rangle: \mathcal{M}(T)\times \mathcal{C}(T)\to \R$
be the bilinear form defined by
$
\langle y,h\rangle:=\int_{T}h(\tau)\ydt
$ 
for $y\in \mathcal{M}(T)$ and $h\in \mathcal{C}(T)$.
We then endow $\mathcal{M}(T)$ with the $\wast$-topology,
which is the minimum topology such that any seminorm
$p_{\A}$ on $\mathcal{M}(T)$ is continuous for any finite subset $\A\subseteq \mathcal{C}(T)$, where 
$p_{\A}:\mathcal{M}(T)\to \R$ is 
defined by
$
p_{\A}(y):=\max_{h\in \A}|\langle y,h\rangle|.
$

Let us here specify the concept of accumulation points and limit points in the sense of the $\wast$-topology.
Let $\{y^k\}$ be a sequence in $\mathcal{M}(T)$ and $y^{\ast}\in \mathcal{M}(T)$.
\begin{enumerate}
\item We call $y^{\ast}$ the $\wast$ limit point of $\{y^k\}$
if for any neighborhood $\mathcal{N}(y^{\ast})$ of $y^{\ast}$ with respect to the $\wast$-topology
there exists an integer $K\ge 0$ such that $y^k\in\mathcal{N}(y^{\ast})$ for any $k\ge K$.
We then say $\{y^k\}$ $\wlyast$ converges to $y^{\ast}$ and often write it as {$\wlim_{k\to \infty}y^k=y^{\ast}$}.
\item
We call $y^{\ast}$ a $\wast$ accumulation point of $\{y^k\}$
if for any integer $K\ge 0$ and neighborhood $\mathcal{N}(y^{\ast})$ of $y^{\ast}$ with respect to the $\wast$-topology
there exists an integer $k\ge K$ such that $y^k\in\mathcal{N}(y^{\ast})$.
\end{enumerate}
\section{Primal-dual path-following method}\label{sec:3}
\subsection{KKT conditions for the SISDP}
In this section, we present the Karush-Kuhn-Tucker (KKT) conditions for the SISDP together with Slater's constraint qualification, abbreviated as SCQ.
Here, SCQ for the SISDP is defined precisely as below:
\begin{definition}
We say that the Slater constraint qualification (SCQ) holds for the SISDP if there exists some $\bar{x}\in \R^n$ such that
$
F(\bar{x})\in S^m_{++}$ and $g(\bar{x},\tau)<0\ (\tau\in T).
$
\end{definition}
\begin{theorem}\label{kkt}
Let $x^{\ast}\in \R^n$ be a local optimal solution of the SISDP\,\eqref{lsisdp}. 
Then, under the SCQ,
there exists some finite Borel-measure $y\in \M(T)$ such that 
\begin{align}
&\nabla f(x^{\ast})+\int_T\nabla_xg(x^{\ast},\tau)\ydt-
\FiV
=0,\label{e1}\\ 
  &{F}(x^{\ast})\circ V=O,\ F(x^{\ast})\in S^m_+,\ V\in S^m_+,\label{e2}\\
&\int_Tg(x^{\ast},\tau)\ydt=0,\ g(x^{\ast},\tau)\le 0\ (\tau\in T),\ y\in\M_+(T), \label{e4}
\end{align}
where 
$V\in S^m$ is a Lagrange multiplier matrix associated with the constraint $F(x)\in S^m_{+}$. 
In particular, there exists some discrete measure $y\in \M_+(T)$ satisfying the above conditions and $|{\rm supp}(y)|\le n$, where
${\rm supp}(y):=\{\tau\in T\mid y(\{\tau\})\neq 0\}$.
Conversely, when $f$ is convex, if the above conditions {\eqref{e1}--\eqref{e4}} hold, then $x^{\ast}$ is an optimum of the SISDP\,\eqref{lsisdp}.   
\end{theorem}
\begin{proof}
Note that $F(x^{\ast})\bullet V=0,\ F(x^{\ast})\in S^m_+$ and $V\in S^m_+$ hold
if and only if $F(x^{\ast})\circ V=O,\ F(x^{\ast})\in S^m_+$, and $V\in S^m_+$.
Then, the claim is proved in a manner similar to \cite[Theorem~2.4]{okuno2012regularized}. 
\hspace{\fill}$\square$
\end{proof}
\noindent {The system \eqref{e1}--\eqref{e4} is called the Karush-Kuhn-Tucker (KKT) conditions for the SISDP\,\eqref{lsisdp}.}
{We call 
$(x,y,V)$ satisfying the KKT conditions \eqref{e1}--\eqref{e4} a KKT point of the SISDP\,\eqref{lsisdp} in particular.}
\subsection{Description of the algorithm}\label{sec:rmu}
In this section, we propose an algorithm for solving the SISDP\,\eqref{lsisdp}, whose 
fundamental framework is analogous to the primal-dual interior point method developed for solving the nonlinear SDP in \cite{yabe}.
It aims to find a KKT point of the SISDP\,\eqref{lsisdp}, i.e., a point satisfying the {optimality} conditions \eqref{e1}--\eqref{e4} for the SISDP\,\eqref{lsisdp}.

{Let us define the} function $R_{\mu}:\R^n\times \mathcal{M}(T)\times S^m_{+}\to \R$ with a parameter $\mu\ge 0$ by
\begin{equation}
R_{\mu}(x,y,V):=
\sqrt{
\theta(x)^2+
\|\phi_1(x,{y},{V})\|^2+\phi_2(x,y)^2+
\|\phi_3(x,V,\mu)\|^2},\notag 
\end{equation}
where
\begin{align}
\theta(x)&:=\max_{\tau\in T}\,\left(g(x,\tau)\right)_+,\vspace{0em}\notag \\
\phi_1(x,{y},{V})&:=\nabla f(x)+{\displaystyle \int_T\nabla_xg(x,\tau)\ydt} -\FiV,\vspace{0em}\notag \\
\phi_2(x,y)&:=\int_Tg(x,\tau)\ydt,\notag\\
\phi_3(x,V,\mu)&:={{\rm svec}}\left(F(x)\circ V-\mu I\right).\notag 
\end{align}
Notice that a point satisfying $R_0(x,y,V)=0$ with $F(x)\in S^m_+$
and $V\in S^m_+$ is nothing but a KKT point of the SISDP\,\eqref{lsisdp}.   
In terms of the function $R_{\mu}$, we define a barrier KKT(BKKT) point by perturbing 
the semi-definite complementarity condition in the KKT conditions\,\eqref{e1}--\eqref{e4}.
\begin{definition}\label{def_bkkt} 
Let $\mu>0$. 
We call $\left(x,y,V\right)\in \R^n\times \mathcal{M}(T){\times S^m}$ a barrier Karush-Kuhn-Tucker (BKKT) point of the SISDP\,\eqref{lsisdp} if 
$R_{\mu}(x,y,V)=0$,\ $y\in \mathcal{M}_+(T)$,\ $F(x)\in S^m_{++}$, $V\in S^m_{++}$.
\end{definition}
Additionally, given a positive parameter $\epsilon$, we define a neighborhood of the BKKT points with barrier parameter $\mu$:
\begin{equation*}
{\mathcal{N}}_{\mu}^{\varepsilon}:=\left\{w:=(x,y,V)\in \R^n\times \mathcal{M}_+(T)\times S^m_{++}\mid R_{\mu}(w)\le \varepsilon,\ F(x)\in S^m_{++}\right\}.
\end{equation*}
The algorithm generates a sequence of 
approximate BKKT points $\{w^k\}$ for the SISDP\,\eqref{lsisdp}
{such that 
$w^k\in \mathcal{N}_{\mu_k}^{\epsilon_k}$} for each $k$
while driving the values of both parameters $\mu_k$ and $\epsilon_k$ to $0$ as $k$ tends to $\infty$.\vspace{0.5em}\\
{\bf Algorithm~1} (Primal-dual path following method)\vspace{0.5em}
\begin{description}
\item[Step 0 (Initial setting):]
Choose an initial iteration point $w^0:=(x^0,y^0,V_0)\in \R^n\times \mathcal{M}_+(T)\times S^m$ such that
$F(x^0)\in S^m_{++}$ and $V_0\in S^m_{++}$.\vspace{-0.0em}
Choose the initial parameters $\mu_0>0$, $\epsilon_0>0$ 
and $\beta\in (0,1)$.
Let $k:=0$.
\item[Step~1 (Stopping rule):]
Stop if 
\begin{equation}
R_{0}(w^k)=0,\ {F(x^k)\in S^m_+,\ V_k\in S^m_+,\ y^k\in \mathcal{M}_+(T).}
\end{equation}
Otherwise, go to Step~2.
\item[Step~{2} (Computing an approximate BKKT point):]
Find {an} approximate BKKT point
$w^{k+1}$ such that 
\begin{equation}
w^{k+1}\in \mathcal{N}_{\mu_{k}}^{\epsilon_{k}}.
\label{bkkt_eq}
\end{equation}
\item[Step~3 (Update):] 
Set $\mu_{k+1}:=\beta \mu_k$ and $\epsilon_{k+1}:=\beta \epsilon_k$.
Let $k:=k+1$. Return to Step~1.
\end{description}
In the recent work\,\cite{okuno2018sc}, the authors propose the interior-point SQP method for computing a BKKT point and show its global convergence property.
If we use the interior-point SQP method as a subroutine to find an approximate BKKT point satisfying condition\,\eqref{bkkt_eq}, Step~2 of Algorithm~1 is well-defined, i.e., such an approximate BKKT point
can be found {in} finitely many steps.
\input{conv.tex}
\input{superlinear.tex}
\input{numeric.tex}

\section{Conclusion}
In this paper, we proposed 
two algorithms for solving the SISDP\,\eqref{lsisdp}: The first one (Algorithm~1) is a primal-dual path following method
designed to find a KKT point of the SISDP by following a path {formed by} BKKT points.
We showed that a sequence generated by the algorithm $\mbox{weakly}^{\ast}$ converges to a KKT point under some mild assumptions. To accelerate local convergence speed,  
the second algorithm (Algorithm~2) integrates
a two-step SQP method into Algorithm~1.
Algorithm~2 solves a sequence of quadratic programs and Newton equations obtained by the local reduction method and Monteiro-Zhang scaling technique, while decreasing the value of the barrier parameter. 

We established two-step superlinear convergence of Algorithm~2 
for the particular case where the AHO-like directions is used. 
As for the cases of the NT and H.K.M directions, we can show a two-step superlinear convergence in a manner analogous to  
\cite[Theorems~3,4]{yamashita2012local}.
Finally, we conducted some numerical experiments to investigate the efficiency of Algorithm~2 by comparing it with the discretization method which solves (nonlinear) SDPs obtained by finite relaxation of the SISDP\,\eqref{lsisdp}. 
In the experiments, we confirmed that the sequences generated by Algorithm~2 actually converged to a KKT point two-step superlinearly.
We also observed that it exhibited the numerical efficiency comparable to the discretization method.
In particular, it worked better in finding highly accurate solutions than the discretization method.

\input{appendix}

\end{document}

%% file: conv.tex
\subsection{Convergence analysis}
In this section, 
we suppose the well-definedness of Step~3 in Algorithm~1 and establish its $\wast$ convergence to KKT points
of SISDP\,\eqref{lsisdp}. Furthermore, we will characterize $\wast$ accumulation points 
of the generated sequence 
more precisely for some special cases.
For the sake of analysis, we assume that Algorithm~1 produces an infinite sequence and further make the following assumptions:\vspace{1.0em}\\
{\bf Assumption~A}\vspace{0.0em}
\begin{enumerate}
\item\label{A2} {The} feasible set of SISDP\,\eqref{lsisdp} is {nonempty and} compact.
\item Slater's constraint qualification holds for SISDP\,\eqref{lsisdp}.
\end{enumerate}
Let $S^{\ast}\subseteq \R^n$ be the optimal solution set of SISDP\,\eqref{lsisdp} and $\bar{v}\in \R$ be
a constant larger than the optimal value of the SISDP.
If $f$ is convex, Assumption~A-\ref{A2} can be replaced with the milder assumption
that $S^{\ast}$ is compact by adding a convex constraint $f(x)\le \bar{v}$ to the SISDP without changing the shape of $S^{\ast}$. 
Under the above assumptions, we first show that the generated sequences $\{x^k\}$ and $\left\{(y^k,V_k)\right\}$ are bounded. 
\begin{proposition}\label{bound1}
Suppose that Assumption A-\ref{A2} holds. Then, any sequence $\{x^k\}$
produced by Algorithm~1 is bounded.
\end{proposition}
\begin{proof}
Denote the feasible set of SISDP\,\eqref{lsisdp} by $\mathcal{F}$ and 
define a proper closed convex function $\phi:\R^n\to \R$ by
$$
\phi(x):={\max}\Big(-\lambda_{\min}(F(x)),\ \max_{\tau\in T}g(x,\tau)\Big).
$$
Since the level set $\{x\in \R^n \mid \phi(x)\le 0\}(=\mathcal{F})$ is compact, 
any level set $\{x\in \R^n \mid \phi(x)\le \eta\}$ with $\eta>0$ is also compact. From \eqref{bkkt_eq} and $\varepsilon_k\le \varepsilon_0$
for all $k$ sufficiently large, 
it is not difficult to show that
$\{x^k\}\subseteq \{x\in \R^n \mid \phi(x)\le \varepsilon_0\}$, where $\varepsilon_0$ is an algorithmic parameter given in Step~0, and thus $\{x^k\}$ is bounded.
\\
\hspace{\fill}$\square$
\end{proof}
 \begin{proposition}\label{prop2}
 Suppose that Assumption~A holds. 
 Then, the generated Lagrange multiplier sequences $\{V_k\}\subseteq S^m_{++}$ {and $\{y^k\}\subseteq M_+(T)$} are bounded.   
 \end{proposition}
 \begin{proof}
 For simplicity of expression,
 denote $\tilde{w}^k:=(V_k,y^k)\in S^m\times \mathcal{M}_+(T)$ and 
 \begin{equation*}
 W_k:=\frac{V_k}{\|\tilde{w}^k\|},\ p^k:=\frac{y^k}{\|\tilde{w}^k\|}
 \end{equation*}
 {where $\|\cdot\|$ is a suitable norm such that
$\|\tilde{w}^k\|^2=\|V_k\|^2+\|y^k\|^2$
on $S^m\times \mathcal{M}(T)$.}
 For contradiction, suppose that there exists a subsequence 
 $\{\tilde{w}^k\}_{k\in K}\subseteq \{\tilde{w}^k\}$
 such that $\|\tilde{w}^k\|\to \infty\ (k\in K\to \infty)$.
 Note that $\{(W_k,p^k)\}$ is bounded.
 Notice also that 
 the corresponding sequence
 $\{x^k\}_{k\in K}$ is bounded from Proposition\,\ref{bound1}.
Recall that any bounded sequence in $\M(T)$ has at least one $\wast$ accumulation point and one can extract a subsequence $\wlyast$ converging to that point. Thanks to this property, without loss of generality we can assume that there exists a point $\left(x^{\ast},W_{\ast},p^{\ast}\right)\in \R^n\times S^m_+\times \M_+(T)$ such that
 {\begin{align*}
 \lim_{k\in K\to \infty}\left(x^k,W_k\right)
 =\left(x^{\ast},W_{\ast}\right),\ \wlim_{k\in K\to \infty} p^k=p^{\ast}.\notag 
 \end{align*}
Note, in particular, that $\|(W_{\ast},p^{\ast})\|=1$, since $\wlim_{k\in K\to\infty}p^k=p^{\ast}$ entails 
the relation that $$
\lim_{k\in K\to \infty}\|p^k\|
=\lim_{k\in K\to \infty}\int_Tdp^k(\tau)=\int_Tdp^{\ast}(\tau)=\|p^{\ast}\|
$$ and therefore 
\begin{align*}
\|(W_{\ast},p^{\ast})\|^2&=\|W_{\ast}\|^2+\|p^{\ast}\|^2\\
                                             &=\lim_{k\in K\to\infty}\left(\|W_{k}\|^2+\|p^{k}\|^2\right)\\
                                             &=1.
\end{align*}
 From \eqref{bkkt_eq}, for {each} $k\ge 1$, we have  
 \begin{align}
 &\left\|\frac{\nabla f(x^k)}{\|\tilde{w}^k\|}-\FiWk
 +\int_T\nabla_x g(x^k,\tau)dp^k(\tau)\right\|\le \frac{\epsilon_{k-1}}{\|\tilde{w}^k\|},\notag \\
 &\left|
 \int_Tg(x^k,\tau)dp^k(\tau)
 \right|\le {\frac{\epsilon_{k-1}}{\|\tilde{w}^k\|}},\ p^k\in\M_+(T),\notag \\
 &\left\|F(x^k)\circ W_k{-\frac{\mu_{k-1}}{\|\tilde{w}^k\|}I}\right\|\le \frac{\epsilon_{k-1}}{\|\tilde{w}^k\|},\ F(x^k)\in S^m_{++},\ W_k\in S^m_{++}.\notag 
 \end{align}
 By letting $k\in K\to \infty$, we obtain 
 \begin{align}
&\FiWast
-\int_T\nabla_x g(x^{\ast},\tau)dp^{\ast}(\tau)=0,\label{cp1} \\
 &\int_Tg(x^{\ast},\tau)dp^{\ast}(\tau)=0,\ p^{\ast}\in\mathcal{M}_+(T),\label{cp4} \\
 &F(x^{\ast})\circ W_{\ast}=O,\ F(x^{\ast})\in S^m_+,\ W_{\ast}\in S^m_+.\label{cp5}
 \end{align} 
 Now, choose a Slater point $\tilde{x}\in \R^n$ arbitrarily and let $\tilde{d}:=\tilde{x}-x^{\ast}$.
Notice here that 
\begin{equation}
F(\tilde{x})\bullet W_{\ast}\ge 0,\ \int_Tg(\tilde{x},\tau)dp^{\ast}(\tau)\le 0, \label{eq:0807}
\end{equation}
since 
\begin{equation}
F(\tilde{x})\in S^m_{++},\ W_{\ast}\in S^m_{+},\ \max_{\tau\in T}g(\tilde{x},\tau)<0,\ p^{\ast}\in \M_+(T). \label{eq:0807-2}
\end{equation}
Then, it holds that
 \begin{eqnarray}
   \lefteqn{F(\tilde{x})\bullet W_{\ast}-\int_Tg(\tilde{x},\tau)dp^{\ast}(\tau)}\notag \\
 &= &F(x^{\ast}+\tilde{d})\bullet W_{\ast}
 -\int_Tg(x^{\ast}+\tilde{d},\tau)dp^{\ast}(\tau)\notag \\ 
 &\le&F(x^{\ast}+\tilde{d})\bullet W_{\ast}-\int_T\left(g(x^{\ast},\tau)
+\nabla_xg(x^{\ast},\tau)^{\top}\tilde{d}\right)dp^{\ast}(\tau)\notag \\
&=&\tilde{d}^{\top}
\FiWast
-\int_T \left(\nabla _xg(x^{\ast},\tau)^{\top}\tilde{d}\right)dp^{\ast}(\tau)\notag \\
 &=&\tilde{d}^{\top}\left(
\FiWast 
-\int_T\nabla_xg(x^{\ast},\tau)dp^{\ast}(\tau)\right)\notag \\
 &=&0, \label{eqa:0807-2}
\end{eqnarray}
where the first inequality holds because $g(x^{\ast},\tau)+\nabla_xg(x^{\ast},\tau)^{\top}\tilde{d}\le g(x^{\ast}+\tilde{d},\tau)\ (\tau\in T)$ by the convexity of $g(\cdot,\tau)$.
Moreover, the third equality is obtained from \eqref{cp4} and the fact that $F(x^{\ast})\bullet W_{\ast}=0$ {by} \eqref{cp5}.
The last equality is due to \eqref{cp1}.
Combining \eqref{eq:0807} and \eqref{eqa:0807-2} implies that 
$F(\tilde{x})\bullet W_{\ast}=0$ and 
$\int_Tg(\tilde{x})dp^{\ast}(\tau)=0$, from which we can conclude
$W_{\ast}=O$ and $p^{\ast}=0$ by using \eqref{eq:0807-2} again.
 However, this contradicts $\|(W_{\ast},p^{\ast})\|=1$.
The proof is complete.}\\
\hspace{\fill}$\square$
\end{proof}
Now, we are ready to establish the global convergence property of Algorithm~1.  
\begin{theorem}\label{thm:0612}
Suppose that Assumption~A holds. Then, the sequence $\{(x^k,y^k,V_k)\}$ 
produced by Algorithm~1 is bounded. 
Let $(x^{\ast},y^{\ast},V_{\ast})\in \R^n\times \M_+(T)\times S^m$ be a $\wast$-accumulation point of $\{(x^k,y^k,V_k)\}$.
Then, $(x^{\ast},y^{\ast},V_{\ast})$ is a KKT point of SISDP\,\eqref{lsisdp}. 
In particular, if $f$ is convex, $x^{\ast}$ is an optimum.
\end{theorem}
\begin{proof}
The boundedness of $\{(x^k,y^k,V_k)\}$ follows from Propositions~{\ref{bound1} and \ref{prop2}}.
It remains to show the second half of the theorem.
We can assume  
$\lim_{k\to \infty}(x^k,V_k)=(x^{\ast},V_{\ast})$ and $\wlim_{k\to\infty}y^k=y^{\ast}$ without loss of generality.
Then, by letting $k\to \infty$ in \eqref{bkkt_eq}, 
we see that the KKT conditions\,\eqref{e1}--\eqref{e4} hold with $V=V_{\ast}$ and $y=y^{\ast}$.
By the second half of Theorem~\ref{kkt}, $x^{\ast}$ is an optimum of SISDP\,\eqref{lsisdp} when $f$ is convex.
\hspace{\fill}$\square$
\end{proof}
Subsequently, let us consider the situation where
the number of elements of ${\rm supp}(y^k)$ is bounded from above through execution of the algorithm.
In this case, we can find a more precise form of the $\wast$ accumulation points of $\{y^k\}$.
To see this, we begin with assuming $|{\rm supp}(y^k)|\le M$ for any $k\ge 0$ with some $M>0$, and consider a sequence $\{t^k\}\subseteq T^M:=\overbrace{T\times \cdots \times T}^{M\mbox{ times}}$ 
with $t^k:=(\tau^k_1,\tau^k_2,\ldots,\tau^k_M)$ such that $t^k$ has all elements of ${\rm supp}(y^k)$ as a sub-vector and
$y^k(\tau^k_i)=0$ if $\tau^k_i\notin {\rm supp}(y^k)$ for $i=1,2,\ldots,M$.
Denote $\zeta^k:=(y^k(\tau_1^k),y^k(\tau_2^k),\ldots,y^k(\tau_M^k))^{\top}\in \R^M_+$ for $k=1,2,\ldots$.
In a manner similar to Proposition\,\ref{prop2}, we can show that $\{t^k\}$ and the accompanying sequence $\{(x^k,V_k)\}$ are bounded and have accumulation points with regard to the norm topology. 
Without loss of generality, we suppose that there exist $(x^{\ast},V_{\ast})\in \R^n\times S^m_+$, $t^{\ast}=(\tau^{\ast}_1,\tau^{\ast}_2,\ldots,\tau^{\ast}_M)\in T^M$, and $\zeta^{\ast}=(\zeta^{\ast}_1,\zeta^{\ast}_2,\ldots,\zeta^{\ast}_M)\in \R^M_+$ such that $\lim_{k\to \infty}(x^k,V_k,\zeta^k,t^k)=(x^{\ast},V_{\ast},\zeta^{\ast},t^{\ast})$. Then we can establish the following theorem concerning the explicit form of the $\wast$-accumulation point of $\{(x^k,y^k,V_k)\}$. In the remainder of the section, we use the notations and symbols introduced in this paragraph.
\begin{theorem}\label{thm:0912}
Denote the {distinct} elements of $\{\tau_1^{\ast},\tau_2^{\ast},\ldots,\tau_M^{\ast}\}$ by $s_1,s_2,\ldots,s_p\in T$, where $p\le M$,  
and define a finite discrete measure $y^{\ast}:\mathcal{B}\to \R_+$ by 
\begin{equation}
y^{\ast}(A):=\sum_{j=1}^p\xi_j^{\ast}\delta_A(s_j)\ \ \ (A\in \mathcal{B}),\label{eq:yast}
\end{equation}
where 
$
 \xi_j^{\ast}:=\sum_{i:\tau_i^{\ast}=s_j}\zeta^{\ast}_i
$ for $j=1,2,\ldots,p$.
Then, $\wlim_{k\to \infty}y^k=y^{\ast}$ holds and $(x^{\ast},y^{\ast},V_{\ast})$ is a KKT point of SISDP\,\eqref{lsisdp}.
\end{theorem}
\begin{proof}
Since the proof is straightforward, we omit it.
\hspace{\fill}$\square$
\end{proof}
Let us end the section with the most concise but practical version for Theorem~\ref{thm:0912}.
Let $y^{\ast}$ be the measure defined by \eqref{eq:yast} and consider the case where
$|{\rm supp}(y^{\ast})|=M$.   
Then, we readily obtain the following corollary from Theorem~\ref{thm:0912}:
\begin{corollary}
Suppose that ${\rm supp}(y^{\ast})=\{\tau_1^{\ast},\tau_2^{\ast},\ldots,\tau_M^{\ast}\}$ and $\tau^{\ast}_i\neq \tau^{\ast}_j$ for any $i\neq j$.
Then,  
$\{y^k\}$ converges to $y^{\ast}$ strongly on $\mathcal{M}(T)$
and $(x^{\ast},y^{\ast},V_{\ast})$ is a KKT point for SISDP\,\eqref{lsisdp}.
\end{corollary}

%% file: superlinear.tex
\section{Two-step superlinearly convergent algorithm}\label{sec:4}
In the section, for the sake of rapid local convergence,
we propose to integrate the local reduction method\,\cite{sip2,pereira2009reduction,okuno2014local,Tanaka}, which is a classical semi-infinite optimization method, with Algorithm~1. 
Throughout this section, we assume that the compact metric space $T$ is a bounded closed set in $\R^q$
formed by finitely many sufficiently smooth inequality constraints.
Also, we often identify $X\in S^m$ with $\svec(X)\in \R^{\frac{m(m+1)}{2}}$.
\subsection{The overall structure of the proposed algorithm}
The proposed method is designed to converge to a KKT point of SISDP\,\eqref{lsisdp} at least two-step superlinearly while satisfying the interior point constraints.
More precisely, it generates a sequence $\left\{w^k\right\}:=\left\{(x^k,y^k,V_k)\right\}\subseteq \R^n\times \mathcal{M}_+(T)\times S^m_{++}$ 
together with two kinds of search directions 
\begin{equation*}
\left\{\Delta_{\frac{1}{2}} w^{k}\right\}:=\left\{(\Delta_{\frac{1}{2}} x^{k},\Delta_{\frac{1}{2}} y^{k},\Delta_{\frac{1}{2}} V_{k})\right\}\mbox{ and }
\left\{\Delta_1w^{k}\right\}:=\left\{\left(\Delta_{1} x^{k},\Delta_{1} y^{k},\Delta_{1} V_{k}\right)\right\}
\end{equation*}
such that
\begin{align}
&\left\|w^{k}+\skh\Delta_{\frac{1}{2}} w^{k}+\sk\Delta_1 w^{k}-w^{\ast}\right\|=o\left(\left\|w^k-w^{\ast}\right\|\right),\notag\\
&F(x^k+\skh \Delta_{\frac{1}{2}}x^k)\in S^m_{++},\ F(x^k+\skh \Delta_{\frac{1}{2}}x^k+\sk\Delta_1x^k)\in S^m_{++},\label{al:1115-1}\\
&V_k+\skh \Delta_{\frac{1}{2}}V_k\in S^m_{++},\ V_k+\skh \Delta_{\frac{1}{2}}V_k+\sk\Delta_1V_k\in S^m_{++}\label{al:1115-2}
\end{align}
for any $k$ sufficiently large, where $w^{\ast}$ is a KKT point satisfying a certain regularity condition and $s_{j}^k\ (j=\fr,1)$ are step-sizes determined by%
\footnote{
For $X\in S^{m}_{++},Y\in S^m$, the eigenvalues of $X^{-1}Y$ are real numbers, and hence $\lambda_{\rm min}(X^{-1}Y)\in \R$.}
\begin{equation}
s_j^{k}=\min(t^{k}_j,u^{k}_j),\label{eq:s0}
\end{equation}
where 
\begin{align*}
t^{k}_j&:=
\begin{cases}
- \displaystyle{\frac{\delta}{\lambda_{\rm min}(F(x^{k+j-\fr})^{-1}\sum_{i=1}^n\Delta_jx^k_iF_i)}}\le \delta
\hspace{1em}&\mbox{if }\lambda_{\rm min}\left(F(x^{k+j-\fr})^{-1}\sum_{i=1}^n\Delta_jx^k_iF_i\right)\le -1\\
1                                                     &\mbox{otherwise}, 
\end{cases}\\
u^{k}_j&:=
\begin{cases}
- \displaystyle{\frac{\delta}{\lambda_{\rm min}(V^{-1}_{k+j-\fr}\Delta V_{k+j-\fr})}}\le \delta\hspace{1em}&\mbox{if }\lambda_{\rm min}\left(V_{k+j-\fr}^{-1}\Delta_jV_{k+j-\fr}\right)\le -1\\
1                                                     &\mbox{otherwise}, 
\end{cases}
\end{align*}
for $j={\frac{1}{2}},1$, where $\delta\in (0,1)$ is a {prescribed} algorithmic constant.
Here, $(x^{k+\fr},V_{k+\fr})$ is defined as
\begin{equation*}
(x^{k+\fr},V_{k+\fr}):=\left(x^k+\skh\Delta_{\fr}x^k,V_k+\skh\Delta_{\fr}V_k\right).
\end{equation*}
By the above choice of the step-sizes, 
$s_{\fr}^k,s_{1}^k\in (0,1]$ holds and 
the interior point constraints \eqref{al:1115-1} and \eqref{al:1115-2} are valid since
\begin{align}
\bar{s}&:={\sup\left\{s\mid \lambda_{\rm min}(X+s\Delta X)\ge 0,s\ge 0\right\}}\notag\\
&=\begin{cases}
-\displaystyle{\frac{1}{\lambda_{\rm min}(X^{-1}\Delta X)}}\ &\mbox{if }\lambda_{\rm min}(X^{-1}\Delta X)<0\\
{\infty}\ &\mbox{otherwise}
\end{cases}\label{al:1116-lam}
\end{align} 
for given $X\in S^m_{++}$ and $\Delta X\in S^m$.
{We remark that if $X+\Delta X\in S^m_{++}$, then $\bar{s}>1
$ and hence
the step-size rule along with \eqref{al:1116-lam} yields
\begin{equation}
\lambda_{\min}(X^{-1}\Delta X)>-1.\label{eq:1206-1}
\end{equation}}
So as to attain fast convergence speed as above, 
we {try to} follow the central path closely by
updating the barrier parameter $\mu_k$ so that $\mu_{k+1}=o(\mu_k)$ and 
solving certain nonlinear systems to have 
the search directions $\Delta_{1}w^k$ and $\Delta_{\frac{1}{2}}w^k$.
When those directions turn out to be unsuccessful, a point near the central path is computed by the interior-point SQP method developed in the recent paper\,\cite{okuno2018sc}.
Before describing the details, we first show the overall structure of the proposed algorithm: \vspace{0.5em}\\
{\bf Algorithm~2} (Superlinearly convergent primal-dual path following method)\vspace{0.5em}
\begin{description}
\item[Step 0 (Initial setting):]
Choose parameters 
\begin{equation*}
  {0<\alpha< 1},\ 0<\beta<1,\ \gamma_1,\gamma_2>0,\ \delta\in (0,1),\ \mu_0>0,\ 0<c\le \frac{1}{\alpha+2}.
\end{equation*}
Set $\varepsilon_0:=\gamma_1\mu_0^{1+\alpha}$.
Choose the initial iteration point ${w^0:=}(x^0,y^0,V_0)\in \R^n\times \mathcal{M}_+(T)\times S^m$ such that $F(x^0)\in S^m_{++}$ and $V_0\in S^m_{++}$. 
Let $k:=0$.
\item[Step~1 (Stopping rule):]
Stop if 
\begin{equation*}
R_{0}({w^k})=0,\ F(x^k)\in S^m_+,\ V_k\in S^m_+,\ y^k\in \mathcal{M}_+(T).
\end{equation*}
Otherwise, go to Step~2.
\item[Step~{2} (Computing an approximate BKKT point):]
Find an approximate BKKT point ${w^{k+1}}\in \mathcal{N}_{\mu_k}^{\epsilon_k}$
by the following procedure:
\begin{description}
\item[Step~2-1:] 
Choose a scaling matrix $P_k$ and obtain $\Delta_{\frac{1}{2}}w^k$ by solving the mixed linear complementarity system\,\eqref{al:0506-1}, \eqref{al:0506-3}, and \eqref{al:0506-2}, which amounts to solving the QP\,\eqref{al:qp}
(see Section\,\ref{sec:QP}) with 
$\mu=\mu_k$, $P=P_k$ and 
$\bar{w}=w^k$.
Compute $\skh$ by \eqref{eq:s0} with $j=\frac{1}{2}$ and set $w^{k+\fr}:=w^k+\skh\Delta_{\fr}w^k$.
\item[Step~2-2:]
Choose a scaling matrix $P_{k+\frac{1}{2}}$.
If the linear equations \eqref{al:1201-1}--\eqref{al:1201-4}
(see Section\,\ref{sec:QP})
with $\mu=\mu_k$, $P=P_{k+\frac{1}{2}}$ and
$\bar{w}=w^{k+\fr}$ are solvable, then set a solution as $\Delta_{1}w^k$
and compute $\sk$ by \eqref{eq:s0} with $j=1$. 
Otherwise, go to Step~2-4.
\item[Step~2-3:] If $w^k_+:=w^{k+\fr}+\sk\Delta_{1}w^k\in \mathcal{N}_{\mu_k}^{\epsilon_k}$, set $w^{k+1}:=w^{k}_+$ and go to Step~3. Otherwise, go to Step~2-4.
\item[Step~2-4:]Find $w^{k+1}\in \mathcal{N}_{\mu_{k}}^{\epsilon_{k}}$ using the interior-point SQP method.
\end{description}
\item[Step~3 (Update):] 
Update the parameters as
\begin{equation}
\mu_{k+1}:=\min\left(\beta\mu_k, \gamma_2\mu_k^{1+c\alpha}\right),
\varepsilon_{k+1}:=\gamma_1\mu_{k+1}^{1+\alpha}.\label{eq:update_mu}
\end{equation}
Set $k:=k+1$ and return to Step~1.
\end{description}
We will discuss the structure of the mixed linear complementarity system\,\eqref{al:0506-1}--\eqref{al:0506-3} 
and the equations
\eqref{al:1201-1}--\eqref{al:1201-4}
in Steps~2-1 and 2-2 
later in Section\,\ref{sec:QP}.
As is confirmed easily, Algorithm~2 is a variant of Algorithm~1.
Hence, by Theorem~\ref{thm:0612}, we ensure its global convergence to a KKT point.
In the subsequent convergence analysis, we will focus on the local convergence {rate} of Algorithm~2.
\subsection{Local reduction technique}\label{sec:local}
We explain the local reduction method to the SISDP\,\eqref{lsisdp} briefly. 
For more details, we refer {the} readers to \cite{sip2,pereira2009reduction,okuno2014local,Tanaka}. 
Suppose that we are standing at a point $\bar{x}\in \R^n$. 
The local reduction method represents 
the semi-infinite region $D:=\{x\in\R^n\mid g(x,\tau)\le 0\ (\tau\in T)\}$ with finitely many inequality constraints locally around $\bar{x}$.
Specifically,
in some open neighborhood of $\bar{x}$, say $U(\bar{x})$, 
it expresses the region $D\cap U(\bar{x})$ as 
$$
D\cap U(\bar{x})=\{x\in U(\bar{x})\mid g(x,\tau^i_{\bx}(x))\le 0\ (i=1,2,\ldots,p(\bar{x}))\}
$$
using smooth implicit functions $\tau^i_{\bx}:U(\bar{x})\to T\ (i=1,2,\ldots,p(\bar{x}))$ {with some nonnegative integer $p(\bar{x})$}.
Then, {SISDP\,\eqref{lsisdp} is locally equivalent to} the problem with finitely many inequality constraints in $U(\bar{x})$, namely,
\begin{align}
 \begin{array}{ll}
  \displaystyle{\mathop{\rm Minimize}_{x\in U(\bar{x})}}  &f(x)
  \vspace{0.5em}\\
  {\rm subject~to} &\hat{g}_i(x):=g(x,\tau^i_{\bx}(x))\le 0\ (i=1,2,\ldots,p(\bar{x})), \\
                   &{F}(x)\in S^m_{+},
 \end{array}\label{reduced_LSISDP}
\end{align}
to which 
standard nonlinear optimization algorithms such as the SQP-type method are conceptually applicable.
In what follows, we clarify the condition under which the functions $\tau^i_{\bx}(\cdot)\ (i=1,2,\ldots,p(\bar{x}))$ and the {open neighborhood} $U(\bar{x})$ exist.
Let us denote by $S(x)$ the set of all local maximizers of $\max_{\tau\in T} g(x,\tau)$ and let 
\begin{equation}
S_{\delta}(x):=\{\tau\in S(x)\mid g(x,\tau)>\max_{\tau\in T} g(x,\tau)-\delta\}
\label{eq:Sdelta}
\end{equation}
for a given constant $\delta>0$.
Moreover, define the nondegeneracy of $\bar{x}$ as follows:
\begin{definition}\label{def:nond}
We say that $\bar{x}$ is nondegenerate for $\max_{\tau\in T}g(\bar{x},\tau)$ and $\delta>0$ if 
$|S_{\delta}(\bar{x})|<\infty$ and the linear independence constraint qualification, the second-order sufficient conditions, and the strict complementarity condition regarding $\max_{\tau\in T}g(\bar{x},\tau)$ hold at any $\tau\in S_{\delta}(\bar{x})$.
\end{definition}
If $\bar{x}$ is nondegenerate, there exist an open neighborhood $U(\bar{x})\subseteq \R^n$, a nonnegative integer
$p(\bar{x}):=|S_{\delta}(\bar{x})|$, and twice continuously differentiable implicit functions $\tau^i_{\bx}(\cdot):U(\bar{x})\to T$
such that $S_{\delta}(x)=\{\tau^i_{\bx}(x)\}_{i=1}^{p(\bar{x})}$ and $\{\tau^i_{\bx}(x)\}_{i=1}^{p(\bar{x})}$ are strict local maximizers
in $\max_{\tau\in T}g(x,\tau)$ for any $x\in U(\bar{x})$.
With those implicit functions, it holds that $\max_{\tau\in T}g(x,\tau)= \max_{1\le i\le p(\bar{x})}\hat{g}_i(x)$ in $U(\bar{x})$ and thus SISDP\,\eqref{lsisdp} and nonlinear SDP\,\eqref{reduced_LSISDP} are equivalent locally.

The functions $\hat{g}_i(\cdot)\ (i=1,2,\ldots,p(\bx))$ 
are convex in $U(\bx)$ when the functions $g(\cdot,\tau)\ (\tau\in T)$ are convex. 
Indeed, for each $i=1,2,\ldots,p(\bx)$, 
there exists some neighborhood $T_i\subseteq T$ of $\tau^i_{\bx}(\bx)$ 
such that $\max_{\tau\in T_i}g(x,\tau)=\hat{g}_i(x)$ holds for any $x\in U(\bx)$.
By noting that $\max_{\tau\in T_i}g(\cdot,\tau)\ (i=1,2,\ldots,p(\bar{x}))$ are convex,
we then ensure the convexity of 
$\hat{g}_i(\cdot)\ (i=1,2,\ldots,p(\bar{x}))$
in $U(\bx)$.

We can compute the values of {$\nabla\tau^i_{\bx}(\bar{x})$ for $i=1,2,\ldots,p(\bar{x})$ by solving a certain linear system derived from the implicit function theorem, from which we further obtain the values of $\nabla\hat{g}_i(\bar{x})$ and $\nabla^2\hat{g}_i(\bar{x})$ for each $i$.}
Thanks to {this} result, 
we acquire the concrete forms of the quadratic programs (QPs) that arise in the SQP iterations
for \eqref{reduced_LSISDP}, although it is difficult in general to have explicit forms of the functions $\tau^i_{\bx}(\cdot)\ (i=1,2,\ldots,p(\bar{x}))$.
\subsection{Computing the directions $\Delta_{\frac{1}{2}} w$ and $\Delta_1 w$}\label{sec:QP}
\subsubsection{First direction $\fDelta w$}\label{sec:gene}
Let $\bar{w}=(\bx,\bar{y},\bar{V})\in \R^n\times \mathcal{M}_+(T)\times S^m_{++}$ be 
the current point such that $F(\bx)\in S^m_{++}$ and $\bx$ is nondegenerate in the sense of Definition\,\ref{def:nond}.
We show that a first search direction $\fDelta w=(\fDelta x,\fDelta y,\fDelta V)\in \R^n\times \mathcal{M}(T)\times S^m$
can be computed through the local reduction method
in a manner similar to the interior-point SQP method proposed in the recent work\,\cite{okuno2018sc}.

To start with, we apply the Monteiro-Zhang scaling to $F(x)$ and $V$, in which
we select a nonsingular matrix $P\in \R^{m\times m}$ and scale
the matrices $F(x)$ and $V$ as 
\begin{align}
&F_P(x):=PF(x)P^{\top}=F_0+\sum_{i=1}^nx_iF_P^i,\ {V}_P:=P^{-\top}V P^{-1},\label{scal}                   
\end{align}
where
$F_P^i:=PF_iP^{\top}$
for $i=0,1,2,\ldots,n$.
Let us consider the reduced NSDP\,\eqref{reduced_LSISDP} with $F(x)\in S^m_{+}$ replaced by $F_P(x)\in S^m_+$, called the scaled NSDP\,\eqref{reduced_LSISDP}.
Since $F(x)\circ V=\mu I,\ F(x)\in S^m_{++},V\in S^m_{++}$ if and only if $F_P(x)\circ V_P=\mu I,\ F_P(x)\in S^m_{++},V_P\in S^m_{++}$ for any $\mu\ge 0$, the KKT (BKKT) conditions of the reduced NSDP\,\eqref{reduced_LSISDP}
are equivalent to those of the scaled NSDP\,\eqref{reduced_LSISDP}. 
Therefore, to produce a search direction, it is natural to solve the following mixed linear complementarity system approximating 
the BKKT system of the scaled NSDP:
\begin{align}
&\nabla f(\bx)+{\nabla_{xx}^2L}(\bx,\by)\fDelta x+\nabla\hat{g}(\bx)(\by+\fDelta y)-
\left(F_P^i\bullet \left(\bV_P+\fDelta V_P\right)\right)_{i=1}^n=0,\label{al:0506-1}\\
&F_P(\bx)\circ(\bV_P+\fDelta{V}_P)+\mathcal{L}_{\bV_P}\sum_{i=1}^n\fDelta x_iF_P^i=\mu I,\label{al:0814-1}\\
&0\le y+\fDelta y\perp \hat{g}(\bx)+\nabla \hat{g}(\bx)^{\top}\fDelta x\le 0.\label{al:0506-3}
\end{align}
In our method, we make a slight modification to the above system. Specifically, we replace the second equation\,\eqref{al:0814-1} with the following equation:
\begin{equation}
F_P(\bx)\circ(\bV_P+\fDelta{V}_P)+\frac{1}{2}{\left(\mathcal{L}_{\bV_P}+\mathcal{L}_{F_P(\bx)}\mathcal{L}_{\bV_P}\mathcal{L}_{F_P(\bx)}^{-1}\right)\sum_{i=1}^n\fDelta x_iF_P^i}=\mu I.\label{al:0506-2}
\end{equation}
The second term of the left hand side 
approximates $\mathcal{L}_{\bV_P}\sum_{i=1}^n\Delta x_iF^i_P$ around a BKKT point. 
Actually, at any BKKT point, 
those two expressions are identical to each other since $\mathcal{L}_{F_P(\bx)}$ and $\mathcal{L}_{\bV_P}$ commute there.
Particularly when choosing a scaling matrix $P$ so that $F_P(\bx)$ and $\bV_P$ commute, \eqref{al:0814-1} and \eqref{al:0506-2} become identical to each other.

{The reason for using the system \eqref{al:0506-1}, \eqref{al:0506-3}, and \eqref{al:0506-2}
is that it can be solved via a KKT system of the following quadratic program (QP):
\begin{align} 
 \begin{array}{ll}
  \displaystyle{\mathop{\rm Minimize}_{\Delta x}}  &\nabla f(\bx)^{\top}\Delta x+
\frac{1}{2}\Delta x^{\top}B_P(\bx,\by,\bV)\Delta x-\mu \xi_P(\bx)^{\top}\Delta x\vspace{0.5em}\\
  {\rm subject~to} &\hat{g}(\bx)+\nabla \hat{g}(\bx)^{\top}\Delta x\le 0, 
 \end{array}\label{al:qp}
\end{align}
where 
$\xi_P(\cdot):=\nabla \log\det F_P(\cdot)=(F^i_P\bullet F_P(\cdot)^{-1})_{i=1}^n$, $\hat{g}(\cdot):=(\hat{g}_1(\cdot),\hat{g}_2(\cdot),\ldots,\hat{g}_{p(\bx)}(\cdot))^{\top}$, and  
\begin{equation}
B_P(x,y,V):=\nabla_{xx}^2{L}(x,y)+{H}_{P}(x,V)\label{eq:0913}
\end{equation}
with ${L}(x,y)$ being the Lagrangian $f(x)+\sum_{i=1}^{p(\bx)}\hat{g}_i(x)y(\tau^i_{\bx}(x))-F(x)\bullet V$
and ${H}_P(x,V)$ being the symmetric matrix whose elements are defined by  
\begin{equation}
\left({H}_P(x,V)\right)_{i,j}:=\frac{1}{2}{F_P^i\bullet\left(\mathcal{L}_{F_P(x)}^{-1}\mathcal{L}_{V_P}
+\mathcal{L}_{V_P}\mathcal{L}_{F_P(x)}^{-1}
\right)F_P^j}\label{eq:HP}
\end{equation}
for $i,j=1,2,\ldots,{n}$.
Note that the linear operator $\mathcal{L}_{F_P(x)}$ is invertible when $F(x)\in S^m_{++}$.
Denote a KKT pair of the QP\,\eqref{al:qp} by $\left(\fDelta x,y+\fDelta y\right)$ and define
\begin{align*}
\fDelta V&:=\mu{F}(\bx)^{-1}-\bV-\sum_{i=1}^n\fDelta x_iP^{\top}\frac{1}{2}{\left(\mathcal{L}_{F_P(\bx)}^{-1}\mathcal{L}_{\bar{V}_P}+\mathcal{L}_{\bar{V}_P}\mathcal{L}_{F_P(\bx)}^{-1}\right)F_P^i}P,\notag\\
\fDelta V_P&:=P^{-\top}\fDelta VP^{-1}.\notag
\end{align*}
Then, we can see that the triple $\left(\fDelta x,\fDelta y, \fDelta V\right)$ solves the system \eqref{al:0506-1}, \eqref{al:0506-3}, and \eqref{al:0506-2}.}

The QP\,\eqref{al:qp} is necessarily feasible 
if the original problem \eqref{lsisdp} is feasible, since the functions 
$\hat{g}_i(\cdot)\ (i=1,2,\ldots,p(\bx))$ are convex as mentioned
above. Furthermore, we have the following property concerning the strong convexity of the objective function of the QP\,\eqref{al:qp}
\begin{proposition}\label{rem:0607}
Suppose that ${F}_P(x)\in S^m_{++}$, ${V}_P\in S^m_{++}$, and $F_1,F_2,\ldots, F_n$ are linearly independent in $S^m$.
Also, suppose that either of the following is true: 
\begin{enumerate}
\item[(i)] $\|F_P(x)\circ {V}_P-\mu I\|\le \theta \mu$ with $0\le \theta<1$;
\item[(ii)] $F_P(x)$ and $V_P$ commute.
\end{enumerate}
Then, 
${H}_P(x,V)$ is positive definite.
Especially, if $f$ is convex and ${\bar{y}}\in \R^{p(\bx)}_{+}$, the objective function of the QP\,\eqref{al:qp} is strongly convex. Therefore, it has a unique optimum. 
\end{proposition}
\begin{proof}
Note that ${F}_P(x)\circ {V}_P\in S^m_{++}$ holds if either of the assumptions (i) and (ii) holds.
Then, the operators $\mathcal{L}_{{F}_P(x)}\mathcal{L}_{V_P}$ and $\mathcal{L}_{{V}_P}\mathcal{L}_{{F}_P(x)}$ are positive definite.  Actually, for any $D\in S^m\setminus \{O\}$,
$D\bullet \mathcal{L}_{{F}_P(x)}\mathcal{L}_{{V}_P}D=
D\bullet\mathcal{L}_{{V}_P}\mathcal{L}_{{F}_P(x)}D
={\rm Tr}(D({F}_P(x)\circ {V}_P)D)>0$.
Then, letting $\Delta F:=\sum_{i=1}^n\Delta x_i{F}_P^i$ and noting
the linear independence of $F_1,F_2,\ldots,$ and $F_n$ in $S^m$, we obtain 
\begin{align*}
2\Delta x^{\top}H_P(x,V)\Delta x&=\Delta F\bullet \left(\mathcal{L}_{F_P(x)}^{-1}\mathcal{L}_{{V}_P}+\mathcal{L}_{V_P}
\mathcal{L}_{F_P(x)}^{-1}\right)\Delta F\\
                         &=\mathcal{L}_{F_P(x)}^{-1}(\Delta F)\bullet \left(\mathcal{L}_{V_P}\mathcal{L}_{F_P(x)}+\mathcal{L}_{F_P(x)}\mathcal{L}_{V_P}\right)\mathcal{L}_{F_P(x)}^{-1}(\Delta F)>0
\end{align*}
for any $\Delta x\neq 0$. We omit the proof for the latter claim.
\end{proof}
Below, we list some particular choices for the scaling matrix $P$ and the corresponding directions ${\fDelta V}$:
\begin{enumerate}
\item[(i)] $P=I$: {In this case,} $F_P(\bx)=F(\bx)$ and
\begin{equation*} 
\fDelta V=\mu{F}(\bx)^{-1}-\overline{V}-\frac{1}{2}{\left(\mathcal{L}_{{F}(\bx)}^{-1}\mathcal{L}_{\bV}
+\mathcal{L}_{\bV}\mathcal{L}_{{F}(\bx)}^{-1}\right)\sum_{i=1}^n\fDelta x_iF_i}.
\end{equation*}  
\item[(ii)] $P=F(\bx)^{-\frac{1}{2}}$:
In this case, ${F}_P(\bx)=I$ and
\begin{equation*}
\fDelta V=\mu{F}(\bx)^{-1}-\bV-\frac{1}{2}{F(\bx)^{-1}\left(\sum_{i=1}^n\fDelta x_iF_i\right)\bV+\bV\left(\sum_{i=1}^n\fDelta x_iF_i\right)F(\bx)^{-1}}.
\end{equation*}
\item[(iii)] $P=W^{-\frac{1}{2}},\ W:=F(\bx)^{\frac{1}{2}}(F(\bx)^{\frac{1}{2}}\bV F(\bx)^{\frac{1}{2}})^{-\frac{1}{2}}F(\bx)^{\frac{1}{2}}$:
In this case, ${F}_P(\bx)=\bV_P$ and
\begin{equation*} 
\fDelta {V}=\mu{F}(\bx)^{-1}-\bV-W^{-1}\left(\sum_{i=1}^n\fDelta x_i{F}_i\right)W^{-1}.
\end{equation*}
\end{enumerate}
The direction $\fDelta V$ obtained as above can be related to the family of Monteiro-Zhang (MZ) directions\, \cite{wolkowicz2012handbook}.
Actually, as for (i), if $F_P(\bar{x})$ and $\overline{V}_P$ commute, the generated direction
can be cast as the Alizadeh-Hareberly-Overton (AHO) direction.
On the other hand, the generated directions 
in (ii) and (iii) are nothing but the
Helmberg-Rendle-Vanderbei-Wolkowicz/Kojima-Shindoh-Hara/Monteiro (HRVW/KSH/M) and Nesterov-Todd (NT) directions, respectively, by themselves.
\subsubsection{Second direction $\oDelta w$}
We next show how to compute the second direction $\oDelta w$ at $\bar{w}+s\fDelta w$.
In a manner similar to $\fDelta w$, 
we may compute
the second direction $\oDelta w$ by solving the QP\,\eqref{al:qp} with $\bar{w}$ 
and $P$ replaced by $\bar{w}+s\fDelta w$ and 
another scaling matrix $\hP\in \R^{m\times m}$, respectively. 
However, by exploiting information associated to $\Delta_{\fr} x$, we can replace the QP with certain linear equations as follows:
Let $J_a(\bx):=\left\{i\in \{1,2,\ldots,p(\bx)\}\mid \hat{g}_i(\bx)+\nabla \hat{g}_i(\bx)^{\top}\fDelta x=0\right\}$.
If the current point $\bx$ is sufficiently close to a KKT point, we can expect that the inequality constraints $\hat{g}_i(x)\le 0\ (i\in J_a(\bx))$ are also active at the KKT point. 
Motivated by this observation, 
we propose to solve the following linear equations 
for $\oDelta w_{\hP}:=(\oDelta x,\oDelta y,\oDelta V_{\hP})$:
\begin{align}
&\nabla f(\hx)+{\nabla_{xx}^2L}(\hx,\hy)\oDelta x+\nabla\hat{g}(\hx)(\hy+\oDelta y)-
\left(F_{\hP}^i\bullet \left(\hV_{\hP}+\oDelta V_{\hP}\right)\right)_{i=1}^n=0,\label{al:1201-1}\\
&F_{\hP}(\hx)\circ(\hV_{\hP}+\oDelta{V}_{\hP})+\frac{1}{2}{\left(\mathcal{L}_{\hV_{\hP}}+\mathcal{L}_{F_{\hP}(\hx)}\mathcal{L}_{\hV_{\hP}}\mathcal{L}_{F_{\hP}(\hx)}^{-1}\right)\sum_{i=1}^n\oDelta x_iF_{\hP}^i}=\mu I,
\label{al:1201-2}\\
&\hat{g}_i(\hx)+\nabla \hat{g}_i(\hx)^{\top}\oDelta x=0\ (i\in J_a(\bx)),\label{al:1201-3}\\
&\hy_i+\oDelta y_i=0\ (i\notin J_a(\bx)),\label{al:1201-4}
\end{align}
where $(\hx,\hy,\hV_{\hP}):=\bar{w}_{\hP}+s\fDelta w_{\hP}$.
{We then set $\oDelta w:=\left(\oDelta x,\oDelta y,\oDelta V\right)$ with $\oDelta V:=\hP^{\top}\oDelta V_{\hP}\hP$.}
If the above linear equations are not solvable or not well-defined
because
$\{\tau^i_{\bx}(\cdot)\}_{i\in J_a(\bx)}\subseteq \{\tau^i_{\hx}(\cdot)\}_{i=1}^{p(\hx)}$ does not hold, i.e., 
the {family of} functions ${\hat{g}}_i(\cdot)\ (i\in J_a(\bx))$ defined at $\bx$ is not valid at $\hx$, then
we skip the above procedure and proceed to the next step.
\subsection{Local convergence analysis}\label{sec:local}
\input{super_conv.tex}


%% file: super_conv.tex
In this section, we focus on the case where the identity matrix is selected as a scaling matrix, i.e., $P=I$. Accordingly, ${F}_P(x)=F(x)$ and ${V}_P=V$ hold throughout the section. For the other cases where 
scaling matrices corresponding to HRVW/KSH/M and NT directions are used (recall (ii) and (iii) in Section\,\ref{sec:gene}), we can also show results similar to the ones given below in a manner analogous to \cite[Theorems~3,4]{yamashita2012local}.

Let $w^{\ast}=(x^{\ast},y^{\ast},V_{\ast})$ be an arbitrary $\wast$-accumulation point of the generated sequence $\{w^k\}$. Recall that $w^{\ast}$ is a KKT point of SISDP\,\eqref{lsisdp} by Theorem\,\ref{thm:0612}.
  Our aim in the section is to examine the convergence rate 
under the assumption that $\lim_{k\to \infty}w^k=w^{\ast}$.  

In what follows, we will make two {sets} of assumptions. Firstly, we assume the following hypotheses concerning the NSDP\,\eqref{reduced_LSISDP} obtained through the local reduction around $x^{\ast}$.
{\begin{flushleft}
{\bf Assumption~B:}
\end{flushleft}
\begin{enumerate}
\item\label{B1} $\lim_{k\to \infty}w^k=w^{\ast}$.
\item  {The point $x^{\ast}$ is nondegenerate in the sense of Definition\,\ref{def:nond}.
Hence, we have NSDP\,\eqref{reduced_LSISDP} with $\bx=x^{\ast}$ together with the implicit functions
$\tau^{i}_{x^{\ast}}(\cdot):U(x^{\ast})\to T\ (i=1,2,\ldots,p(x^{\ast}))$. }
\item\label{B3}   
For sufficiently large $k$, 
$$\left\{\tau^i_{x^k}(\cdot)\right\}_{i=1}^{p(x^k)}=\left\{\tau^i_{x^{\ast}}(\cdot)\right\}_{i=1}^{p(x^{\ast})}\ \mbox{and }{\rm supp}(y^k)\subseteq\left\{\tau^1_{x^{\ast}}(x^k),\tau^2_{x^{\ast}}(x^k),\ldots,\tau^{p(x^{\ast})}_{x^{\ast}}(x^k)\right\}.$$
\end{enumerate}
Hereafter, according to Assumption~B-\ref{B3}, we identify Borel measures $y^k,y^{\ast}\in \mathcal{M}(T)$ with some vectors in $\R^{p(x^{\ast})}$ as follows:
\begin{align*}
y^k&=(y^k_1,y^k_2,\ldots,y^k_{p(x^{\ast})})^{\top}\in \R^{p(x^{\ast})},\ y^k_i:=y^k(
\tau_{x^{\ast}}^i(x^k))\ (i=1,2,\ldots,p(x^{\ast})),\\
y^{\ast}&=(y^{\ast}_1,y^{\ast}_2,\ldots,y^{\ast}_{p(x^{\ast})})^{\top}\in \R^{p(x^{\ast})},\ y^{\ast}_i:=y^{\ast}(\tau_{x^{\ast}}^i(x^{\ast}))\ (i=1,2,\ldots,p(x^{\ast})).
\end{align*}}
Let us define some functions and notations.
Denote by $I_a(x^{\ast})$
the set of indices corresponding to the active inequality constraints at $x^{\ast}$ among $\hat{g}_1(x)\le 0,\ldots,\hat{g}_{p(x^{\ast})}(x)\le 0$, i.e.,
$I_a(x^{\ast}):=\{i\in \{1,2,\ldots,p(x^{\ast})\}\mid \hat{g}_i(x^{\ast})=0\}$.
For a fixed barrier parameter $\mu>0$, we consider the function $\Phi_{\mu}:\R^n\times \R^{|I_a(x^{\ast})|}\times \R^{m(m+1)/2}\to \R^d$ with 
$
d:={n+|I_a(x^{\ast})|+m(m+1)/2}$ defined by 
\begin{equation}
\Phi_{\mu}(\tilde{w}):=\begin{pmatrix}
&\nabla f(x)+{\displaystyle\sum_{i\in I_a(x^{\ast})}y_i\nabla \hat{g}_i(x)}-
\FiV\\
&\left(\hat{g}_i(x)\right)_{i\in I_a(x^{\ast})}\\
&{\rm svec}\left(F(x)\circ V-\mu I\right)
\end{pmatrix},\label{eq:Phidef}
\end{equation}
where we write $\tilde{w}:=(x,\tilde{y},\svec(V))
\in \R^n\times\R^{|I_a(x^{\ast})|}\times \R^{m(m+1)/2}$ with $\tilde{y}=(y_i)_{i\in I_a(x^{\ast})}$.
Denote the Jacobian of $\Phi_{\mu}(\cdot)$
by $\mathcal{J}\Phi_{\mu}(\cdot)$, that is,  
$$
\mathcal{J}{\Phi}_{\mu}(\tilde{w}):=\begin{pmatrix}
\nabla^2f(x)+{\displaystyle\sum_{i\in I_a(x^{\ast})}y_i\nabla^2\hat{g}_i(x)}

&(\nabla\hat{g}_i(x))_{i\in I_a(x^{\ast})}&
- \begin{pmatrix}{\rm svec}(F_1)^{\top}\\
                                                                                
                                                                                                \vdots \\
                                                                                        {\rm svec}(F_n)^{\top}
                                                                                      \end{pmatrix}\\
(\nabla\hat{g}_i(x))_{i\in I_a(x^{\ast})}^{\top} &0&0\\
 {\rm svec}(\mathcal{L}_{V}(F_1))\cdots {\rm svec}(\mathcal{L}_{V}(F_n))                            &0  &         {\mathcal{T}_{F(x)}}                \\
\end{pmatrix},
$$
where ${\mathcal{T}_X}\in \R^{m(m+1)/2\times m(m+1)/2}$ is defined as the matrix such that 
$
{\mathcal{T}_X}{\rm svec}(Y):={\rm svec}\left(\mathcal{L}_{X}(Y)\right)
$
for any $X,Y\in S^m$.
It is worth mentioning that $\mathcal{J}\Phi_{\mu}=\mathcal{J}\Phi_{0}$ holds for any $\mu\ge 0$.
Then, {$\tilde{w}^{\ast}:=(x^{\ast},\tilde{y}^{\ast},\svec(V_{\ast}))$}
solves $\Phi_{0}(\tilde{w})=0$.
Correspondingly, let us define the function $Q:\R^n\times \R^{|I_a(x^{\ast})|}\times \R^{m(m+1)/2}\to \R^{d\times d}$ by
replacing the block consisting of {${\rm svec}(\mathcal{L}_V(F_i))$, $i=1,2,\ldots,n$} in $\mathcal{J}\Phi_{\mu}(\tilde{w})$ with 
the matrix with columns
$v_1,v_2,\cdots,v_n$, 
where 
$$
v_i:=\frac{1}{2}{\rm svec}\left({\mathcal{L}_{V}F_i+\mathcal{L}_{F(x)}\mathcal{L}_{V}\mathcal{L}_{F(x)}^{-1}F_i}\right)\ \ (i=1,2,\ldots,n).
$$

Next, for convenience, we also define the following index set for each $k$:
\begin{align}
J_a^k:=\left\{i\in \{1,2,\ldots,p(x^{\ast})\}\mid \hat{g}_i(x^k)+\nabla \hat{g}_i(x^k)^{\top}\Delta_{\frac{1}{2}}x^k=0\right\}.
\label{al:jak1}
\end{align}
We additionally make the second set of assumptions:
\begin{flushleft}
{\bf Assumption~C:}
\end{flushleft}
\begin{enumerate}
\item 
The functions $f$ and $g(\cdot,\tau)\ (\tau\in T)$ are three times continuously differentiable.
\item \label{identify}
The active inequality constraints at $x^{\ast}$ are eventually identified in the sense that  
$J_a^k=I_a(x^{\ast})$ holds for any $k$ sufficiently large.
\item\label{sc} The strict complementarity condition holds for the semi-definite {constraint} and {the} inequality constraints, i.e., 
$$
F(\bxopt)+{V}_{\ast}\in S^m_{++},\ -\hat{g}_i(x^{\ast})+y^{\ast}_i>0\ (i=1,2,\ldots,p(x^{\ast})).
$$
\item\label{nonsing} {The Jacobian $\mathcal{J}\Phi_{0}(\tilde{w}^{\ast})$ is nonsingular.}
\end{enumerate}
We can show that Assumption~C-\ref{identify} holds true when the objective function $f$ is convex, although we omit the proof here.
Also, in a manner analogous to the proof of \cite[Theorem~1]{yamashita2012local}, it is not difficult to verify Assumption~C-\ref{nonsing} under suitable regularity conditions.

In view of the implicit function theorem,
we can ensure the existence of the central path converging to ${w}^{\ast}$ under Assumption~C-\ref{nonsing}. 
Specifically, there exist some $\bar{\mu}>0$ and a smooth curve $\tilde{w}(\cdot): (0,\bar{\mu}]\to U(x^{\ast})\times \R^{|I_a(x^{\ast})|}\times \R^{m(m+1)/2}$
such that 
$
\lim_{\mu\to 0+}\tilde{w}(\mu)=\tilde{w}^{\ast}$ and  
$\tilde{w}(\mu)$ represents the BKKT point with the barrier parameter $\mu\in (0,\bar{\mu}]$.

Recall the definition\,\eqref{al:jak1} of $J_a^k$. Then, 
Assumption~C-\ref{identify}
{along with the complementarity condition \eqref{al:0506-3} with $\bar{w}=w^k$}
yields that, for all $k$ sufficiently large,
$\hat{g}_i(x^k)+\nabla \hat{g}_i(x^k)^{\top}\Delta_{\frac{1}{2}}x^k<0$
and $y^k_i+\Delta_{\fr}y^k_i=0$
for $i\in \{1,2,\ldots,p(x^{\ast})\}\setminus I_a(x^{\ast})$.
Similarly, by \eqref{al:1201-4} with $\bx=x^k$ and $J_a(\bx)$ replaced by
$J_a^k$, i.e., $I_a(x^{\ast})$ under Assumption~C-\ref{identify}, we have
$y^{k+\frac{1}{2}}_i+\Delta_{1}y_i^{k}=0\ \ (i\in \{1,2,\ldots,p(x^{\ast})\}\setminus I_a(x^{\ast}))$.
{Therefore}, we can reduce the system\,\eqref{al:0506-1}, \eqref{al:0506-3}, and \eqref{al:0506-2}
with $P=I$ and $(\mu,\bar{w})=(\mu_k,w^{k})$ and the equations\,\eqref{al:1201-1}--\eqref{al:1201-4}
with $\hat{P}=I$ and $(\mu,\hw)=(\mu_k,w^{k+\fr})$
to the following equations for $j=0$ and $j=\fr$, respectively:
\begin{align}
&\Phi_{\mu_k}(\tilde{w}^{k+j})+Q(\tilde{w}^{k+j})\Delta_{j+\frac{1}{2}}\tilde{w}=0,\label{al:1105-1} 
\end{align}
where
$\Delta_{j+\fr} \tilde{w}:=\left(\Delta_{j+\fr} x,\Delta_{j+\fr}\tilde{y},{\rm svec}\left(\Delta_{j+\fr} V\right)\right)^{\top}$.
\subsubsection{Technical results}
{In this section, we provide two useful propositions before entering the essential part of the convergence analysis.
See the Appendix for the proofs.
The first proposition is associated with the limiting behavior of the operator $\mathcal{L}_{F(x^k)}\mathcal{L}_{V_k}\mathcal{L}_{F(x^k)}^{-1}$ as $k$ tends to  $\infty$.
\begin{proposition}\label{prop:0511}
Let $(X_{\ast},Y_{\ast})\in S^m_+\times S^m_+$ satisfy the strict complementarity condition that $X_{\ast}{\circ}Y_{\ast}=O$ and $X_{\ast}+Y_{\ast}\in S^m_{++}$. 
Let $\{\mu_r\}\subseteq \R_{++}$ and $\{(X_r,Y_r)\}\subseteq S^m_{++}\times S^m_{++}$ be sequences  such that $\lim_{r\to\infty}\mu_r=0$, $\lim_{r\to\infty}(X_r,Y_r)=(X_{\ast},Y_{\ast})$ and $\|X_r\circ Y_r-\mu_rI\|_F=O(\mu_r^{1+\zeta})$ with $\zeta>0$. 
Then,
$\|\mathcal{L}_{X_r}\mathcal{L}_{Y_r}\mathcal{L}_{X_r}^{-1}-\mathcal{L}_{Y_r}\|_2=O(\mu_r^{\zeta})$ and thus 
$\lim_{r\to \infty}\mathcal{L}_{X_r}\mathcal{L}_{Y_r}\mathcal{L}_{X_r}^{-1}=\mathcal{L}_{Y_{\ast}}$,
where $\|\cdot\|_2$ denotes the operator norm, namely, 
for any linear operator $\mathcal{T}:S^m\to S^m$, $\|\mathcal{T}\|_2:=\sup_{\|X\|_F=1}\|\mathcal{T}(X)\|_F$.
\end{proposition}
{The next proposition will be useful in proving that 
the interior point constraints $F(x^{k+j-\fr}+s_k^{j}\jDelta x^k)\in S^m_{++}$ and $V_{k+j-\fr}+s_k^{j}\jDelta V_k\in S^m_{++}$ 
eventually hold true with $s_k^{j}=1$ for $j=\fr,1$.}
\begin{proposition}\label{lem:1203-2}
Let $0<\zeta<1$ and $\{(X_r,Y_r)\}\subseteq S^m_{++}\times S^m_{++}$, 
$\{(\Delta X_r,\Delta Y_r)\}\subseteq S^m\times S^m$, and $\{\mu_r\}\subseteq \R_{++}$ be 
sequences such that $\lim_{r\to \infty}\mu_r=0$, 
 \begin{align}
 &\|\Delta X_r\circ \Delta Y_r\|_F=O(\mu_r^2),\label{al:1202-1}\\
 &\|X_r\circ Y_r-\mu_rI\|_F=O(\mu_r^{1+\zeta}).\label{al:1202-2} 
 \end{align}
Moreover, let $0<\hzeta<1$ and $\{\hmu_r\}\subseteq \R_{++}$ be a sequence such that $\lim_{r\to\infty}\hmu_r=0$,
\begin{align}
&\|Z_r-\hmu_r I\|_F=O(\hmu_r^{1+\hzeta}), \label{al:1202-3}\\
&\mu_r^2=o(\hmu_r),\label{al:1202-7}
\end{align}
where $Z_r:=X_r\circ Y_r+X_r\circ \Delta Y_r+Y_r\circ \Delta X_r$.
Then, we have $X_r+\Delta X_r\in S^m_{++}$ and $Y_r+\Delta Y_r\in S^m_{++}$ for any sufficiently large $r$.
 \end{proposition}}
\subsubsection{Main convergence results}
In this section, we provide the main convergence results for the proposed algorithm.
For the sake of analysis, we choose a parameter $\tilde{c}$ such that
\begin{equation}
\frac{1}{2}<\tilde{c}<\frac{1-c}{1+c\alpha}<1, \label{eq:choice_c}
\end{equation}
where $c$ and $\alpha$ are algorithmic parameters selected in Step~0.
We can ensure that such $\tilde{c}$ exists if the parameter $c$ is selected as in Step~0 of Algorithm~2. 
In terms of $\tilde{c}$, let us define the parameter sequence $\{\tilde{\epsilon}_{k}\}$ by 
\begin{equation}
\tilde{\epsilon}_{k}:=\gamma_1\mu_{k}^{1+\tilde{c}\alpha}\label{eq:1228-1}
\end{equation}
for each $k$.
Note that the second inequality in \eqref{eq:choice_c} implies $(1+c\alpha)(1+\tilde{c}\alpha)<1+\alpha$.
Then, 
from \eqref{eq:update_mu}, we have
\begin{equation}
\tilde{\epsilon}_k>\epsilon_k\label{eq:1228-2}
\end{equation}
for all $k$ large enough.
Furthermore, the update rule \eqref{eq:update_mu} of $\{\mu_k\}$ and $\{\epsilon_k\}$ yields
\begin{equation}
\mu_k=\gamma_2\mu_{k-1}^{1+c\alpha},\ \epsilon_k=
\gamma_1\gamma_2^{1+\alpha}\mu_{k-1}^{(1+c\alpha)(1+\alpha)} \label{eq:1201large}
\end{equation}
for all $k$ sufficiently large. Hereafter, we assume that the iteration number $k$ is so large that \eqref{eq:1228-2} and \eqref{eq:1201large} hold. 

To show the final theorem concerning two-step superlinear convergence (see Theorem~\ref{thm:1124}), 
we prove the following two propositions:
\begin{proposition}\label{prop:0606}
Suppose that Assumptions~B and C hold.
We have 
\begin{enumerate}
\item\label{item:0606-item2} the full step-size $s^k_{\fr}=1$ is eventually adopted in Step~2-1, 
\item\label{prop:0606-item1}
$w^{k+\fr}\in\mathcal{N}_{\mu_{k}}^{\tepsk}$, i.e., ${R}_{\mu_k}(w^{k+\fr})\le {\tepsk}$, $F(x^{k+\fr})\in S^m_{++}$,
and $V_{k+\fr}\in S^m_{++}$ for all $k$ sufficiently large, and
\item\label{prop:0606-1} $\|w^{k+\fr}-w^{\ast}\|=
O\left(\|w^k-w^{\ast}\|^{1+c\alpha}\right)$.
\end{enumerate}
\end{proposition}
\begin{proposition}\label{prop:0606-2}
Suppose that Assumptions~B and C hold.
We have 
\begin{enumerate}
\item\label{item:0606-2-1} the full step-size $\sk=1$ is eventually adopted in Step~2-2,
\item\label{item:0606-2} $w^{k+\fr}+\Delta_1w^k\in \mathcal{N}_{{\tblumuk}}^{\epsilon_k}$, i.e.,
${R}_{{\tblumuk}}(w^{k+\fr}+\Delta_1w^k)\le \epsilon_k=\gamma_1\mu_k^{1+\alpha}$, $F(x^{k+\fr}+\Delta_1x^k)\in S^m_{++}$, and
$V_{k+\fr}+\Delta_1V_k\in S^m_{++}$ for all $k$ sufficiently large and
\item $\|w^{k+\fr}+\Delta_1w^k-w^{\ast}\|=O\left(\|w^{k+\fr}-w^{\ast}\|\right)$.
\end{enumerate}
\end{proposition}
Items~\ref{item:0606-item2} and \ref{prop:0606-item1} of Proposition\,\ref{prop:0606} mean
that $w^{k+\fr}=w^k+\fDelta w^k$ eventually holds
and it is accommodated by
$\mathcal{N}_{\mu_k}^{\tilde{\epsilon}_k}$, which is larger than $\mathcal{N}_{\mu_{k}}^{\epsilon_k}$ since
$\tepsk>\epsilon_k=\gamma_1\mu_k^{1+\alpha}$ for $k$ sufficiently large by \eqref{eq:1228-2}.  
On the other hand, items\,\ref{item:0606-2-1} {and} \ref{item:0606-2} of Proposition\,\ref{prop:0606-2} indicate that $w^{k+\fr}+\oDelta w^k$ is necessarily accepted by 
the targeted neighborhood $\mathcal{N}_{\mu_k}^{\epsilon_k}$ for all $k$ sufficiently large. Hence, the condition in Step~2-3 is eventually satisfied with $s^k_1=1$.

In what follows, we devote ourselves to prove
the above two propositions. 
To begin with, we give some lemmas
that help to show Proposition\,\ref{prop:0606}.
The following lemma is concerned with the convergence speed of $\mu_{k-1}$, $\|\Phi_{0}(\tilde{w}^k)\|$, and $\|{w}^k-{w}^{\ast}\|$.
\begin{lemma}\label{prop:0604-1}
Suppose that Assumptions~B and C hold.
Then, 
we have 
$\|w^k-w^{\ast}\|=\|\tilde{w}^k-\tilde{w}^{\ast}\|$ for sufficiently large $k$
and 
$\mu_{k-1}=\Theta(\|\Phi_{0}(\tilde{w}^k)\|)=\Theta(\|\tilde{w}^k-\tilde{w}^{\ast}\|)=\Theta(\|{w}^k-{w}^{\ast}\|)$.
\end{lemma}
\begin{proof}
See Appendix.
\hspace{\fill}{$\square$}
\end{proof}
\begin{lemma}\label{rem:0517}
Suppose that Assumptions~B and C hold. Then, 
\begin{enumerate}
\item 
{we have}
$$\|\mathcal{J}\Phi_0(\tilde{w}^{k})-Q(\tilde{w}^k)\|_F=
\sqrt{\sum_{i=1}^n\left\|
\frac{1}{2}{\left(\mathcal{L}_{F(x^k)}\mathcal{L}_{V_k}\mathcal{L}_{F(x^k)}^{-1}-\mathcal{L}_{V_k}\right)F_i}
\right\|_F^2}
=
O(\mu_{k-1}^{\alpha})$$
and hence $\lim_{k\to \infty}Q(\tilde{w}^k)=\mathcal{J}\Phi_0(\tilde{w}^{\ast})$;
\item\label{rem:0517-2} $Q(\tilde{w}^k)$ is nonsingular for sufficiently large $k$ and $\{Q(\tilde{w}^k)^{-1}\}$ is bounded.
\end{enumerate}
\end{lemma}
\begin{proof}
Notice that $\|F(x^k)\circ V_k-\mu_{k-1}I\|\le \epsilon_{k-1}=\gamma_1\mu_{k-1}^{1+\alpha}$ by $w^k\in \mathcal{N}_{\tbkpre}^{\epsilon_{k-1}}$ and \eqref{eq:update_mu}.
In addition, note that 
$\lim_{k\to \infty}(F(x^k),V_k)=(F(x^{\ast}),V_{\ast})\in S^m_{+}\times S^m_+$ and 
$\lim_{k\to\infty}F(x^k)+V_k=F(x^{\ast})+V_{\ast}\in S^m_{++}$ 
by Assumptions\,B-\ref{B1} and C-\ref{sc}. Then, 
Proposition\,\ref{prop:0511}
with $\{X_r\}$, $\{Y_r\}$, $\{\mu_r\}$, and $\zeta$ replaced by $\{F(x^k)\}$, $\{V_k\}$, $\{\mu_{k-1}\}$, and $\alpha$, respectively,
yields
\begin{equation*}
\|\mathcal{L}_{{F}(x^k)}\mathcal{L}_{{V}_k}\mathcal{L}_{{F}(x^k)}^{-1}-\mathcal{L}_{{V}_k}\|_2=O(\tbkpre^{\alpha}),
\end{equation*}
which further implies 
\begin{equation}
\|\mathcal{J}\Phi_0(\tilde{w}^{k})-Q(\tilde{w}^k)\|_F=
\sqrt{\sum_{i=1}^n\left\|
\frac{1}{2}{\left(\mathcal{L}_{F(x^k)}\mathcal{L}_{V_k}\mathcal{L}_{F(x^k)}^{-1}-\mathcal{L}_{V_k}\right)F_i}
\right\|_F^2}
=O(\tbkpre^{\alpha}),\notag 
\end{equation}
where the first equality is a direct consequence of 
the fact $\|\svec(X)\|^2=\|X\|_F^2\ (X\in S^m)$ and
the forms of $\mathcal{J}\Phi_0(\tilde{w}^k)$ and $Q(\tilde{w}^k)$.
To prove item~\ref{rem:0517-2},  
recall that $\mathcal{J}\Phi_0(\tilde{w}^{\ast})$ is nonsingular from Assumption~C-\ref{nonsing}. 
Then, since $\lim_{k\to \infty}Q(\tilde{w}^k)=\mathcal{J}\Phi_0(\tilde{w}^{\ast})$ from item~1,
$Q(\tilde{w}^k)$ is nonsingular for all $k$ sufficiently large. 
In addition, we obtain the boundedness of $\{Q(\tilde{w}^k)^{-1}\}$
as $\lim_{k\to \infty}Q(\tilde{w}^k)^{-1}=\mathcal{J}\Phi_0(\tilde{w}^{\ast})^{-1}$ holds. The proof is complete.
\hspace{\fill}{$\square$}\end{proof} 
We are now ready to prove Proposition\,\ref{prop:0606} using Lemmas\,\ref{prop:0604-1} and \ref{rem:0517}.
\begin{description}
\item[\bf Proof of Proposition\,\ref{prop:0606}:]
\end{description}
Recall that $\tilde{w}^{k}=(x^{k},\tilde{y}^{k},\svec(V_{k}))$ with $\tilde{y}^{k}=(y^{k}_i)_{i\in I_a(x^{\ast})}$
and $\fDelta\tilde{w}^{k}=(\fDelta x^{k},\fDelta \tilde{y}^{k},\svec(\fDelta V_{k}))$.
For simplicity of expression, we suppose $I_a(x^{\ast})=\{1,2,\ldots,p(x^{\ast})\}$, which implies $\tilde{w}^{k}=w^k$
and $\fDelta \tilde{w}^k=\fDelta w^k$ for sufficiently large $k$ and $\tilde{w}^{\ast}=w^{\ast}$. It is not difficult to extend the subsequent analysis to the more general case of $I_a(x^{\ast})\subseteq \{1,2,\ldots,p(x^{\ast})\}$.
Hereafter, we write 
\begin{equation}
Q_k:=Q({w}^k),\ \mathcal{J}_k:=\mathcal{J}\Phi_0({w}^k)\label{eq:1228-3}
\end{equation}
for each $k$.

From item\,\ref{rem:0517-2} of Lemma\,\ref{rem:0517} together with \eqref{al:1105-1}, we readily see that 
there exists some $M>0$ such that
\begin{equation}
\|Q_k^{-1}\|_F\le M\label{eq:1205}
\end{equation}
and 
\begin{align}
&\Delta_{\frac{1}{2}}{w}^k=-Q_k^{-1}\Phi_{\mu_k}({w}^k),
\label{eq:1105-214-2}\\
&\|\Delta_{\frac{1}{2}}{w}^k\|\le \|Q_k^{-1}\|_F\|\Phi_{\mu_k}({w}^k)\|=O(\|\Phi_{\mu_k}({w}^k)\|),
\label{eq:1105-2149}
\end{align}
where the equality in \eqref{eq:1105-2149} follows from \eqref{eq:1205}.
Especially, combining the above with Lemma~\ref{prop:0604-1} implies that
\begin{equation}
\max\left(\|\Delta_{\frac{1}{2}}x^k\|,\|\Delta_{\frac{1}{2}}V_k\|\right)\le \|
\Delta_{\frac{1}{2}}{w}^k
\|=O(\mu_{k-1}).\label{eq:1110-1}
\end{equation}
Notice that $\|\Phi_{\mu}({w}^{\ast})\|=\|\mu I\|_F=\mu\sqrt{m}$ for any $\mu\ge 0$ and 
\begin{align}
\|\Phi_{{\tblumuk}}({w}^k)\|
&\le \|\Phi_{0}({w}^k)\|+{\tblumuk}\sqrt{m}\notag \\
&=\|\Phi_{0}({w}^k)\|+o(\mu_{k-1})\notag \\
&=O(\|w^k-w^{\ast}\|)\label{al:1203-3}\\
&=O(\tbkpre),\label{al:1203-4}
\end{align}
where 
the first equality follow from \eqref{eq:1201large}
and the last two equalities are derived from Lemma~\ref{prop:0604-1}.
Additionally, by item~1 of Lemma\,\ref{rem:0517}, we have 
\begin{equation}
\|\mathcal{J}_k-Q_k\|_F=
\sqrt{\sum_{i=1}^n\left\|
\frac{1}{2}{\left(\mathcal{L}_{F(x^k)}\mathcal{L}_{V_k}\mathcal{L}_{F(x^k)}^{-1}-\mathcal{L}_{V_k}\right)F_i}
\right\|_F^2}=O(\tbkpre^{\alpha}).\label{eq:1016-s}
\end{equation}\vspace{0.5em}
{\rm 1.} It suffices to show that 
\begin{equation}
F(x^k+\fDelta x^k)=F(x^k)+\fDelta F_k\in S^m_{++},\ V_k+\fDelta V_k\in S^m_{++}\label{eq:1202-1426}
\end{equation}
for any $k$ sufficiently large, where $\fDelta F_k:=\sum_{i=1}^n\fDelta x^k_iF_i$.
In fact, if these conditions hold, by \eqref{eq:1206-1} with $(X,\Delta X)=(F(x^k),\Delta_{\fr}F_k)$
and $(X,\Delta X)=(V_{k},\Delta_{\fr}V_k)$, we see that $t^k_{\fr}=u^k_{\fr}=1$ and thus $s^k_{\fr}=1$ from \eqref{eq:s0}.
From 
$w^k\in\mathcal{N}_{\tbkpre}^{\epsilon_{k-1}}$
and $\epsilon_{k-1}=\gamma_1\mu_{k-1}^{1+\alpha}$,
we have  
\begin{equation}
\|F(x^k)\circ V_k-\mu_{k-1} I\|_F=O(\mu_{k-1}^{1+\alpha}).\label{eq:1013-1}
\end{equation}
Denote
\begin{equation}
\Gamma_k:=\frac{1}{2}{
{\displaystyle\sum_{i=1}^n}\Delta_{\frac{1}{2}}x_i^k\left(\mathcal{L}_{{F}(x^k)}\mathcal{L}_{{V}_k}\mathcal{L}_{{F}(x^k)}^{-1}-\mathcal{L}_{V_k}\right)F_i}.\label{al:1103-2}
\end{equation}
In view of \eqref{eq:1110-1} and \eqref{eq:1016-s}, we have
\begin{align}
&\|\Gamma_k\|_F\le \|\mathcal{J}_k-Q_k\|_F\|\Delta_{\fr}x^k\|=O(\mu_{k-1}^{1+\alpha}),\label{eq:0612-2}\\
&\|\Delta_{\frac{1}{2}}F_k\circ\Delta_{\frac{1}{2}}V_k\|_F=O(\|\Delta_{\frac{1}{2}}w^k\|^2)
=
O(\mu_{k-1}^2).\label{al:1107}
\end{align}
{By rearranging \eqref{al:0506-2} with $\bar{w}=w^k$, $\mu=\mu_{k}$, and $P=I$ in terms of $\Gamma_k$, we obtain 
$$
F(x^k)\circ V_k+
\Delta_{\frac{1}{2}}F_k\circ V_k+\Delta_{\frac{1}{2}}V_k\circ F(x^k)-\mu_{k} I=-\Gamma_k,
$$
which together with \eqref{eq:0612-2} and \eqref{eq:1201large} implies 
\begin{align}
\|F(x^k)\circ V_k+
\fDelta F_k\circ V_k+\fDelta V_k\circ F(x^k)-\mu_{k} I\|_F&=O(\mu_{k-1}^{1+\alpha})\notag\\ 
                                                         &=O(\mu_k^{1+\frac{(1-c)\alpha}{1+c\alpha}}).                      \label{eq:1203-1304}
\end{align}
Since $\alpha\in (0,1)$ and $c\in (0,1)$,
using \eqref{eq:1201large} again, we obtain
\begin{equation}
\mu_{k-1}^2=o(\mu_k).\label{eq:1206-2031}
\end{equation}
In Proposition\,\ref{lem:1203-2}, replace $\{X_r\}$, $\{Y_r\}$, $\{\Delta X_r\}$, $\{\Delta Y_r\}$, $\{(\mu_r,\hmu_r)\}$, and $(\zeta,\hzeta)$
by $\{F(x^k)\}$, $\{V_k\}$, $\{\fDelta F_k\}$, $\{\fDelta V_k\}$, $\{(\mu_{k-1},\mu_k)\}$, and $(\alpha,\frac{(1-c)\alpha}{1+c\alpha})$, respectively.
Then, the relations\,\eqref{eq:1013-1}, \eqref{al:1107}, \eqref{eq:1203-1304}, and \eqref{eq:1206-2031} correspond to conditions\,\eqref{al:1202-1}--\eqref{al:1202-7}.
We thus have \eqref{eq:1202-1426} by Proposition\,\ref{lem:1203-2}.}\vspace{0.5em}\\
{\rm 2.} We have only to show ${R}_{{\tblumuk}}(w^{k+\fr})\le \tilde{\epsilon}_k$.
To start with, we note that from item\,\ref{item:0606-item2}, $\skh=1$ for all $k$ sufficiently large.
Then, the value of $\|\Phi_{{\tblumuk}}({w}^k+\skh \Delta_{\frac{1}{2}}{w}^k)\|$ is evaluated as follows: 
\begin{align}
\|\Phi_{{\tblumuk}}({w}^k+\Delta_{\frac{1}{2}}{w}^k)\|
&\le
\|\Phi_{{\tblumuk}}({w}^k)+\mathcal{J}_k\Delta_{\frac{1}{2}}{w}^k\|+O(\|\Delta_{\frac{1}{2}}{w}^k\|^2)\notag\\
&=\|\Phi_{\mu_k}({w}^k)-\mathcal{J}_kQ_k^{-1}\Phi_{\mu_k}({w}^k)\|+O(\mu_{k-1}^2)\notag\\
&=\|(Q_k-\mathcal{J}_k)Q_k^{-1}\Phi_{\mu_k}({w}^k)\|+O(\mu_{k-1}^{2})\notag\\
&\le \|\mathcal{J}_k-Q_k\|_F\|Q_k^{-1}\|_F\|\Phi_{\mu_k}({w}^k)\|+O(\mu_{k-1}^{2})\notag\\
&=O(\mu_{k-1}^{1+\alpha})+O(\mu_{k-1}^2)\notag\\
&=O(\mu_{k-1}^{1+\alpha}),\label{al:1016-1}
\end{align}
where
the first inequality follows from \eqref{eq:1228-3}
and $\mathcal{J}_k=\mathcal{J}\Phi_{\mu_k}(w^k)$, the first equality comes from \eqref{eq:1105-214-2} and \eqref{eq:1110-1}, and the third equality is derived from \eqref{al:1203-4}, \eqref{eq:1016-s}, and \eqref{eq:1205}.
From \eqref{al:1016-1}, we further obtain
\begin{align*}
\left|\sum_{i=1}^{p(x^{\ast})}(y^k_i+\Delta_{\frac{1}{2}} y^k_i)\hat{g}_i(x^k+\Delta_{\frac{1}{2}}x^k)\right|
&\le p(x^{\ast})\|{w}^k+\Delta_{\frac{1}{2}} {w}^k\|\|\Phi_{{\tblumuk}}({w}^k+\Delta_{\frac{1}{2}}{w}^k)\|\notag \\
&=O(\|\Phi_{{\tblumuk}}({w}^k+\Delta_{\frac{1}{2}}{w}^k)\|)\\
&=O(\mu_{k-1}^{1+\alpha})
\end{align*}
and
\begin{align*}
\max_{1\le i\le p(x^{\ast})}(\hat{g}_i(x^k+\Delta_{\frac{1}{2}}x^k))_+
&=\max_{i\in I_a(x^{\ast})}(\hat{g}_i(x^k+\Delta_{\frac{1}{2}}x^k))_+\\
&=O(\|\Phi_{{\tblumuk}}({w}^k+\Delta_{\frac{1}{2}}{w}^k)\|)\\
&=O(\mu_{k-1}^{1+\alpha}).
\end{align*}
Combining these facts, we get ${R}_{{\tblumuk}}(w^k+\Delta_{\frac{1}{2}}w^k)=O(\mu_{k-1}^{1+\alpha})$.
Since 
\eqref{eq:1228-1} and \eqref{eq:choice_c} together yield $\tilde{\epsilon}_k=\gamma_1\gamma_2^{1+\tilde{c}\alpha}\mu_{k-1}^{(1+c\alpha)(1+\tilde{c}\alpha)}$ and
$(1+c\alpha)(1+\tilde{c}\alpha)<1+\alpha$,
we conclude ${R}_{{\tblumuk}}(w^k+\Delta_{\frac{1}{2}}w^k)\le\tilde{\epsilon}_{k}$
for any $k$ sufficiently large. \vspace{0.5em}\\
{\rm 3.} Recall that 
$\{\|Q_k^{-1}\|\}$ is bounded by \eqref{eq:1205} and $\tbkpre=\Theta(\|{w}^k-{w}^{\ast}\|)$ from Lemma\,\ref{prop:0604-1}. 
It follows that
\begin{align}
\|{w}^k+&\Delta_{\frac{1}{2}}{w}^k-{w}^{\ast}\|\notag\\
&=
\|{w}^k-Q_k^{-1}\Phi_{\mu_k}(w^k)-{w}^{\ast}\|
\notag\\
&\le \|{w}^k-Q_k^{-1}\Phi_{0}(w^k)-{w}^{\ast}\|+\mu_{k}\|Q_k^{-1}\|_F\sqrt{m}\notag \\
&\le \|Q_k^{-1}\|_F\|Q_k({w}^k-w^{\ast})-\Phi_{0}(w^k)+\Phi_0(w^{\ast})\|+
\gamma_2\mu_{k-1}^{1+c\alpha}\|Q_k^{-1}\|_F\sqrt{m}\notag \\
&=\|Q_k^{-1}\|_F\|\mathcal{J}_k({w}^k-w^{\ast})+(Q_k-\mathcal{J}_k)({w}^k-w^{\ast})-\Phi_{0}(w^k)+\Phi_0(w^{\ast})\|+
O(\mu_{k-1}^{1+c\alpha})\notag \\
&\le \|Q_k^{-1}\|_F\|\mathcal{J}_k({w}^k-{w}^{\ast})-\Phi_{0}({w}^k)+\Phi_{0}({w}^{\ast})\|\notag \\
&\hspace{6em}+\|Q_k^{-1}\|_F\|\mathcal{J}_k-Q_k\|_F\|{w}^k-{w}^{\ast}\|+O(\|w^k-w^{\ast}\|^{1+c\alpha})\notag\\
&=O(\|{w}^k-{w}^{\ast}\|^2)+
O(\|{w}^k-{w}^{\ast}\|^{1+\alpha})
+O(\|{w}^k-{w}^{\ast}\|^{1+c\alpha})\notag \\
&=O(\|{w}^k-{w}^{\ast}\|^{1+c\alpha}),\notag  
\end{align}
where 
the first equality follows from \eqref{eq:1105-214-2}, 
the first inequality comes from $\Phi_{0}(w^k)-\Phi_{\mu_k}(w^k)=(0,0,\svec(\mu_kI))$ (see \eqref{eq:Phidef}),
the second inequality is derived from $\Phi_0(w^{\ast})=0$ and \eqref{eq:1201large}, and 
the third equality is due to \eqref{eq:1016-s} and Lemma\,\ref{prop:0604-1}.
Thus, the desired conclusion is obtained. 
\hspace{\fill}$\square$\vspace{0.5em}\\
We next enter the phase of proving Proposition\,\ref{prop:0606-2}. 
First, let us observe several properties obtained from Proposition\,\ref{prop:0606}.  
Note that item~\ref{prop:0606-1} of Proposition\,\ref{prop:0606} implies $\lim_{k\to \infty}{w}^{k+\fr}={w}^{\ast}$, and thus
\begin{equation}
\lim_{k\to \infty}\tilde{w}^{k+\fr}=\tilde{w}^{\ast}.\label{eq:1203-1616}
\end{equation}
By \eqref{eq:1203-1616}, 
$x^{k+\fr}$ is sufficiently close to $x^{\ast}$ for all $k$ large enough.
Then, from the implicit function theorem we have
$\{\tau_{x^{k+\fr}}^i(\cdot)\}_{i=1}^{p(x^{k+\fr})}=\{\tau_{x^{\ast}}^i(\cdot)\}_{i=1}^{p(x^{\ast})}$, 
while $\{\tau_{x^{k}}^i(\cdot)\}_{i=1}^{p(x^{k})}=\{\tau_{x^{\ast}}^i(\cdot)\}_{i=1}^{p(x^{\ast})}$ {also} holds for all $k$ sufficiently large because of Assumption~B-\ref{B1}.
Thus, 
the sets of implicit functions defined at $x^k$ and $x^{k+\fr}$ are identical, that is to say, $\{\tau_{x^{k+\fr}}^i(\cdot)\}_{i=1}^{p(x^{k+\fr})}=\{\tau_{x^{k}}^i(\cdot)\}_{i=1}^{p(x^{k})}$ holds. 
Furthermore, 
noting $w^{k+\fr}\in\mathcal{N}_{\mu_{k}}^{\tepsk}$, we can show that
$\|\mathcal{J}\Phi_0(\tilde{w}^{k+\fr})-Q(\tilde{w}^{k+\fr})\|_F
=O(\mu_{k}^{\tilde{c}\alpha})$ in a manner similar to item~1 of Lemma\,\ref{rem:0517}.
Thus, it holds that 
$\lim_{k\to\infty}Q(\tilde{w}^{k+\fr})=\mathcal{J}\Phi_0(\tilde{w}^{\ast})$
from \eqref{eq:1203-1616}, which together with Assumption\,C-\ref{nonsing} implies the nonsingularity of $Q(\tilde{w}^{k+\fr})$
for all $k$ sufficiently large
and the boundedness of $\{Q(\tilde{w}^{k+\fr})^{-1}\}$.

The above observations are summarized in the following lemma:
{\begin{lemma}\label{lem:12031219}
Suppose that Assumptions~B and C hold.
Then, we have
\begin{enumerate}
\item\label{lem:12031219-2} the functions $\hat{g}_i(\cdot)\ (i=1,2,\ldots,p(x^k))$ 
defined at $x^k$ are also valid at $x^{k+\fr}$ and hence the linear equations\,\eqref{al:1201-1}--\eqref{al:1201-4} 
are well-defined, and $Q(\tilde{w}^{k+\fr})$ is nonsingular for all $k$ sufficiently large. In addition,  
\item $\|\mathcal{J}\Phi_0(\tilde{w}^{k+\fr})-Q(\tilde{w}^{k+\fr})\|_F
=
O(\mu_{k}^{\tilde{c}\alpha}),$ and 
\item\label{lem:12031219-3} $\{Q(\tilde{w}^{k+\fr})^{-1}\}$ is bounded.
\end{enumerate}
\end{lemma}
Furthermore, in a manner similar to Lemma\,\ref{prop:0604-1}, we can derive the following result in view of $w^{k+\fr}\in \mathcal{N}_{\mu_k}^{\tilde{\epsilon}_k}$:
\begin{lemma}\label{prop:0604-2}
Suppose that Assumptions~B and C hold.
Then, 
we have 
$\|w^{k+\fr}-w^{\ast}\|=\|\tilde{w}^{k+\fr}-\tilde{w}^{\ast}\|$ for sufficiently large $k$
and 
$\mu_{k}=\Theta(\|\Phi_{0}(\tilde{w}^{k+\fr})\|)=\Theta(\|\tilde{w}^{k+\fr}-\tilde{w}^{\ast}\|)=\Theta(\|{w}^{k+\fr}-{w}^{\ast}\|)$.
\end{lemma}
\begin{proof}
The proof is obtained in a manner analogous to Lemma\,\ref{prop:0604-1}. 
\hspace{\fill}{$\square$}\end{proof}}
We are now ready to prove Proposition\,\ref{prop:0606-2}.
Its proof seems quite similar to Propositions\,\ref{prop:0606}.
However, we do not omit it since there are some significant differences. 
For example, the proof of item~\ref{item:0606-2} of Proposition\,\ref{prop:0606-2} relies on
the condition $\tilde{c}>\frac{1}{2}$ in \eqref{eq:choice_c}.
\begin{description}
\item[{\bf Proof of Proposition\,\ref{prop:0606-2}}:]
\end{description}
Like the proof of Proposition\,\ref{prop:0606},
for simplicity of expression, we suppose $I_a(x^{\ast})=\{1,2,\ldots,p(x^{\ast})\}$, which implies $\tilde{w}^{k+\fr}=w^{k+\fr}$
and $\oDelta w^k=\oDelta \tilde{w}^k$ for sufficiently large $k$ and $\tilde{w}^{\ast}=w^{\ast}$. 
We also write 
\begin{equation*}
Q_{k+\fr}:=Q(w^{k+\fr}),\ \mathcal{J}_{k+\frac{1}{2}}:=\mathcal{J}\Phi_{\mu_k}({w}^k+\Delta_{\frac{1}{2}}{w}^k)=\mathcal{J}\Phi_0({w}^k+\Delta_{\frac{1}{2}}{w}^k).
\end{equation*}
By item~2 of Proposition\,\ref{prop:0606}, we have, for any $k$ sufficiently large,  
\begin{equation}
\|\Phi_{\mu_k}(w^k)\|=O(\mu_k^{1+\tilde{c}\alpha})\label{eq:0815-1}
\end{equation}
By \eqref{al:1105-1}
with $j=\fr$ and item~\ref{lem:12031219-2} of Lemma\,\ref{lem:12031219}, we have 
\begin{equation}
\Delta_1w^k=-Q_{k+\frac{1}{2}}^{-1}\Phi_{\mu_k}({w}^{k+\frac{1}{2}}), \label{eq:1203-2155} 
\end{equation}
which together with \eqref{eq:0815-1},  
item~\ref{lem:12031219-3} of Lemma\,\ref{lem:12031219}, and Lemma\,\ref{prop:0604-2}
implies
\begin{equation}
\|\Delta_1w^k\|=O(\|\Phi_{\mu_k}({w}^{k+\frac{1}{2}})\|)=O(\mu_k^{1+\tilde{c}\alpha}).\label{eq:1203-1641}
\end{equation}
Combining this with Lemma\,\ref{prop:0604-2} yields 
\begin{equation}
\|\Delta_1w^k\|=O(\|{w}^{k+\frac{1}{2}}-{w}^{\ast}\|^{1+\tilde{c}\alpha}).\label{eq:1015-0}
\end{equation}
{Note that, from Lemma\,\ref{lem:12031219}, we have
\begin{equation}
\|\mathcal{J}_{k+\fr}-Q_{k+\fr}\|=O(\mu_k^{\tilde{c}\alpha}).\label{eq:0814-1}
\end{equation}} \vspace{0.5em}\\
{\rm 1.} As in the proof of item~\ref{item:0606-item2} of Proposition\,\ref{prop:0606},
it suffices to show that 
\begin{equation}
F(x^{k+\fr}+\Delta_1x^k)=F(x^{k+\fr})+\oDelta F_k\in S^m_{++},\ V_{k+\fr}+\Delta_1V_k\in S^m_{++}\label{eq:1116-1913}
\end{equation}
for all $k$ sufficiently large, where $\oDelta F_k:=\sum_{i=1}^n\oDelta x^k_i F_i$.
Since
$w^{k+\fr}\in\mathcal{N}_{\mu_k}^{\tilde{\epsilon}_{k}}$ by Proposition\,\ref{prop:0606} and $\tilde{\epsilon}_k=\gamma_1\mu_k^{1+\tilde{c}\alpha}$, we obtain 
\begin{equation}
\|F(x^{k+\frac{1}{2}})\circ V_{k+\fr}-{\tblumuk}I\|_F=O(\mu_k^{1+\tilde{c}\alpha}). \label{eq:1015-4}
\end{equation}
{The expression}\,\eqref{eq:1203-1641} implies
\begin{equation}
\|\Delta_{1}F_k\circ\Delta_{1}V_k\|_F=O(\|\Delta_1w^k\|^2)=O(\mu_k^2).\label{eq:1015-5}
\end{equation}
Moreover, by \eqref{al:1201-2} {with $\hat{P}=I,\hw=w^{k+\fr}$, and $\mu=\mu_k$,} it holds that
\begin{equation}
\|F(x^{k+\fr})\circ V_{k+\fr}+\oDelta F_k\circ V_{k+\fr}+F(x^{k+\fr})\circ \oDelta V_k-\mu_k I\|_F=0.\label{eq:1203-2139}
\end{equation}
In Proposition\,\ref{lem:1203-2}, replace $\{X_r\}$, $\{Y_r\}$, $\{\Delta X_r\}$, $\{\Delta Y_r\}$, $\{(\mu_r,\hmu_r)\}$, and $(\zeta,\hzeta)$ by 
$\{F(x^{k+\fr})\}$, $\{V_{k+\fr}\}$, $\{\oDelta F_k\}$, $\{\oDelta V_k\}$, $\{(\mu_k,\mu_k)\}$, and $(\tilde{c}\alpha,\alpha)$, respectively.
Then, in view of \eqref{eq:1015-4}, \eqref{eq:1015-5}, and \eqref{eq:1203-2139}, we can verify the conditions\,\eqref{al:1202-1}--\eqref{al:1202-7}.
We thus obtain \eqref{eq:1116-1913} and conclude the desired result.\vspace{0.5em}\\ 
{\rm 2.} We show only ${R}_{{\tblumuk}}(w^{k+\fr}+\Delta_1w^k)\le \epsilon_k=\gamma_1\mu_k^{1+\alpha}$ for $k$ sufficiently large.
The remaining part is obvious from \eqref{eq:1116-1913}.
We first note that 
\begin{align}
\|\Phi_{{\tblumuk}}({w}^{k+\frac{1}{2}}+\Delta_{1}{w}^k)\|
&\le \|\Phi_{{\tblumuk}}({w}^{k+\frac{1}{2}})+\mathcal{J}_{k+\frac{1}{2}}\Delta_{1}{w}^k\|+O(\|\Delta_{1}{w}^k\|^2)\notag\\
&\le \|\Phi_{{\tblumuk}}({w}^{k+\frac{1}{2}})+Q_{k+\frac{1}{2}}\Delta_{1}{w}^k\| \notag\\
&\hspace{2em}+\|\mathcal{J}_{k+\frac{1}{2}}-Q_{k+\fr}\|\|\Delta_{1}{w}^k\|
+O(\|\Delta_{1}{w}^k\|^2)\notag\\
&=O\left(\mu_k^{\tilde{c}\alpha}\|\Delta_{1}{w}^k\|\right)+O(\|\Delta_{1}{w}^k\|^2)\notag \\
&=O(\mu_k^{1+2\tilde{c}\alpha}), \notag 
\end{align}
where
the first equality follows from \eqref{al:1105-1}, \eqref{eq:1203-2155}, and \eqref{eq:0814-1}
and the second equality is obtained from \eqref{eq:1203-1641}.
Then, in a manner similar to Proposition\,\ref{prop:0606}, we can show
${R}_{\mu_k}(w^{k+\frac{1}{2}}+\Delta_1w^k)=O(\mu_k^{1+2\tilde{c}\alpha})$, which together with
$\tilde{c}>\frac{1}{2}$
from \eqref{eq:choice_c} implies ${R}_{{\tblumuk}}(w^{k+\frac{1}{2}}+\Delta_1w^k)\le \epsilon_k=\gamma_1\mu_k^{1+\alpha}$ for all $k$ sufficiently large. \vspace{0.5em}\\
{\rm 3.} We have 
\begin{align}
\|{w}^{k+\frac{1}{2}}+\Delta_{1}{w}^k-{w}^{\ast}\|&\le \|{w}^{k+\frac{1}{2}}-w^{\ast}\|+\|\Delta_{1}{w}^k\|\notag\\
&=\|{w}^{k+\frac{1}{2}}-w^{\ast}\|+O(\|{w}^{k+\frac{1}{2}}-{w}^{\ast}\|^{1+\tilde{c}\alpha})\notag \\
&=O(\|{w}^{k+\frac{1}{2}}-{w}^{\ast}\|),\notag  
\end{align}
where the first equality follows from \eqref{eq:1015-0}.\hspace{\fill}{$\square$} \vspace{0.5em}\\
Combining Propositions\,\ref{prop:0606} and \ref{prop:0606-2},
we get the following two-step superlinear convergence result.
\begin{theorem}\label{thm:1124}
Suppose that Assumptions~B and C hold.
Then, the update in Step~2-3 is eventually adopted and we have 
\begin{equation}
\|w^{k+1}-w^{\ast}\|=O\left(\|w^k-w^{\ast}\|^{1+c\alpha}\right). \label{eq:0606}
\end{equation}
Hence, $\{w^k\}$ converges to $w^{\ast}$ two-step superlinearly with the order of $1+c\alpha\in \left(1,\frac{4}{3}\right)$.
\end{theorem}
\begin{proof}
From item~\ref{item:0606-2} of Proposition\,\ref{prop:0606-2}, $w^{k+1}=w^{k+\fr}+\Delta_1w^k$ holds for any $k$ sufficiently large, that is to say,
the update in
Step~2-3 is eventually accepted.  Using Propositions\,\ref{prop:0606} and \ref{prop:0606-2} again yields
\begin{equation*}
\|w^{k+\fr}+\Delta_1w^k-w^{\ast}\|=O(\|w^{k+\fr}-w^{\ast}\|)=O(\|w^k-w^{\ast}\|^{1+c\alpha}).
\end{equation*}
We thus confirm \eqref{eq:0606}.
Finally, since the parameter $c$ is chosen so that
$0<c\le \frac{1}{\alpha+2}$, we have
$$1< 1+c\alpha\le 1+\frac{\alpha}{\alpha+2}=2-\frac{2}{\alpha+2}< \frac{4}{3}.$$
The proof is complete.
\hspace{\fill}{$\square$}
\end{proof}

%% file: numeric.tex
\section{Numerical experiments}\label{sec:5}
In this section, we conduct some numerical experiments to demonstrate the efficiency of the primal-dual path following method (Algorithm~2) by solving two kinds of SISDPs with a one-dimensional index set of the form $T=[T_{\rm min},T_{\rm max}]$: The first one is a linear SISDP where all functions are affine with respect to $x$; the second one is an SISDP with a nonlinear objective function.
Throughout the section, we identify a symmetric matrix variable $X\in S^m$ with a vector variable $x:=(x_{11},x_{12},\ldots,x_{1m},x_{12},x_{22},\ldots,x_{mm})^{\top}\in\R^{\frac{m(m+1)}{2}}$ through
$$
X=\begin{pmatrix}
             x_{11}&x_{12}&\ldots&x_{1m}\\
             x_{12}&x_{22}&\ldots&x_{2m}\\
           \vdots&\vdots&\ddots&\vdots\\
             x_{1m}&x_{2m}&\ldots&x_{mm}      
    \end{pmatrix}.
$$ 
The program was coded in MATLAB~R2012a and run on a machine with Intel(R) Xeon(R) CPU E5-1620 v3@3.50GHz and 10.24GB RAM. 
We compute the scaling matrices for the NT direction according to \cite[Section~4.1]{todd1998nesterov}.
As for SISDPs with a nonlinear objective function,
the matrix $B_P$ in the quadratic program\,\eqref{al:qp} is not necessarily positive-definite.
So as to assure its positive definiteness, we modified $B_P$ by lifting its negative eigenvalues to $1$.
Let $\bar{x}$ be a current point and $\{\tau_{\bar{x}}^i(\cdot)\}_{i=1}^{p(\bar{x})}$
be the set of implicit functions defined in \eqref{reduced_LSISDP}. 
As for the set $S_{\delta}(\bar{x})$ defined by \eqref{eq:Sdelta}, 
we set $\delta := 10^{-1}$ and put $N+1$ grids 
$\{s_1,s_2,\ldots,s_{N+1}\}$ on $T$ 
uniformly with $N:=100$. 
To specify the set $S_{\delta}(\bar{x})$, we apply Newton's method combined with the projection onto $T$
for the problem $\max_{\tau\in T}g(\bar{x},\tau)$
starting from each of the local maximizers $\bar{s}$ of $\max\{g(\bar{x},s)\mid s=s_1,s_2,\ldots,s_{N+1}\}$ such that 
$g(\bar{x},\bar{s}) > \max_{1\le i\le N+1}g(\bar{x},s_i) - \delta$.
Let $\bar{y}\in \R^{p(\bar{x})}_+$ be a current estimate of Lagrange multiplier vector associated with the inequality constraints $g(x,\tau_{\bar{x}}^i(x))\le 0\ (i=1,2,\ldots,p(\bar{x}))$. 
As $\bar{x}$ moves to $\bar{x}+\Delta \bar{x}$, we trace the value of the implicit function $\tau_{\bar{x}}^i$ for each $i=1,2,\ldots,p(\bar{x})$, namely, we identify $\tau_{\bar{x}}^i(\bar{x}+\Delta \bar{x})$ with an element in $S_{\delta}(\bar{x}+\Delta \bar{x})$
to examine the correspondence
between $\bar{y}_i\ (i=1,2,\ldots,p(\bar{x}))$ and the inequality constraints ${g}(x,\tau_{\bar{x}+\Delta \bar{x}}^j(x))\le 0\ (1\le j\le p(\bar{x}+\Delta \bar{x}))$. 
For this purpose, for each element $\tau^j_{\bar{x}_+}(\bar{x}_+)\in S_{\delta}(\bar{x}_+)$ with $\bar{x}_+:=\bar{x}+\Delta \bar{x}$, we search $S_{\delta}(\bar{x})=\{\tau_{\bar{x}}^1(\bar{x}),\tau_{\bar{x}}^2(\bar{x}),\ldots,\tau_{\bar{x}}^{p(\bar{x})}(\bar{x})\}$
for an index $\tilde{i}\in \{1,2,\ldots,p(\bar{x})\}$ such that 
$\|\tau^j_{\bar{x}_+}(\bar{x}_+)-\tau_{\bar{x}}^{\tilde{i}}(\bar{x})-\nabla \tau_{\bar{x}}^{\tilde{i}}(\bar{x})^{\top}\Delta \bar{x}\|
(\approx\|\tau^j_{\bar{x}_+}(\bar{x}_+)-\tau_{\bar{x}}^{\tilde{i}}(\bar{x}_+)
\|)
\le 10^{-1}$. 
If it is found, we regard ${\tau}_{\bar{x}_+}^j(\bar{x}_+)$ as $\tau_{\bar{x}}^{\tilde{i}}(\bar{x}_+)$.
Otherwise, we treat ${\tau}_{\bar{x}_+}^j(\cdot)$ as the implicit function that newly appears at $\bar{x}_+$, and set zero to be the Lagrange multiplier for the inequality constraint $g(x,\tau^j_{\bar{x}_+}(x))\le 0$.

Next, we explain how each step of the algorithm is implemented.
In Step~0, we set
$$\gamma_1 =\sqrt{\frac{m(m+1)}{2}},\ \gamma_2=5,\ c = \frac{1}{2.99},\ \alpha = 0.99,\ \beta = 0.8.$$
As for starting points, we set $y^0 = (1,1)^{\top}, V_0=m I$, and $\mu_0=1$, while $x^0$ is chosen so that $X^0 = m^{-1}I$ for linear SISDPs, and $x^0 =0$ is chosen for SISDPs with a nonlinear objective function.
In Step~1, we terminate the algorithm if 
$\mu_{k+1}<10^{-10}$ or  the value of 
the function $R_0$ is less than $10^{-8}$, where $R_0$ is 
the function $R_{\mu}$ with $\mu =0$ defined in Section\,\ref{sec:rmu}.
%
In Step~2.4, we implement the interior-point SQP-type method proposed in \cite{okuno2018sc} by using
the implementation details described therein. 
In Step~3, for the sake of numerical stability, we set $\varepsilon_{k+1}:=\max(10^{-7},\gamma_1\mu_{k+1}^{1+\alpha})$.
For $X\in S^m_{++}$ and $Y\in S^m$, we compute $\mathcal{L}_X^{-1}Y$ by solving the linear equation $\mathcal{L}_XZ=Y$
{for} $Z\in S^m$ with the Matlab built-in solver \texttt{lyap2}. We moreover use \texttt{quadprog} {to solve} quadratic programs in Step~2-1.

For the sake of comparison, we also implement a discretization method that solves finitely relaxed SISDPs sequentially until an approximate feasible solution is obtained. More precisely, 
for solving the SISDP\,\eqref{lsisdp}, we use the following discretization algorithm:
\begin{description}
\item[Step~0:]  Choose an initial index set $T_0\subseteq T$ with
$|T_0|<+\infty$. Choose $\theta>0$. Set $r:=0$. 
\item[Step~1:] Get a KKT point $x^r$ of the finitely relaxed SISDP with $T$ replaced by $T_r$.
\item[Step~2:] Find $\bar{\tau} \in T$
such that $g(x^r,\bar{\tau}) > \theta$ and set $T_{r + 1}:=T_r\cup \{\bar{\tau}\}$.
If such a point does not exist in $T$, terminate the algorithm.  
\item[Step~3:] Increment $r$ by one and return to Step~1. 
\end{description}
In Step~0, we choose $T_0=\{T_{\rm min},T_{\rm max}\}$. In Step~2, 
{to find such a $\bar{\tau}\in T$ we solve $\max_{\tau\in T}g(x^r,\tau)$ by applying Newton's method with a starting point $s\in {\rm argmax}\{
g(x^k,s)\mid s = s_1,s_2,\ldots,s_{N+1}\}$, where $\{s_1,s_2,\ldots,s_{N+1}\}$ is the set of grids defined earlier in this section. 
\footnote{
There is no theoretical guarantee for global optimality of $\tau$ thus found.
In practice, however, we may expect to have a global optimum by setting $N$ large enough.}
We set $\theta:=10^{-6}$.
\subsection{Linear SISDPs}
In this section, we consider the linear SISDP\,\eqref{lsisdp}, called {LSISDP for short}. 
{Specifically, we} solve the following problem taken from \cite[Section~4.2]{li2004solution}:
\begin{align}
\begin{array}{rcl}
\displaystyle{\mathop{\rm Maximize}_{X\in S^m}}& &A_0\bullet X\\
\mbox{subject to}& & A(\tau)\bullet X\ge 0\ (\tau\in T)\\
            & & I\bullet X = 1\\ 
            & & X\in S^m_+,
\end{array}\label{eig_semi}
\end{align}
where $A_0\in S^m$ and $A:T\to S^m$ is a symmetric matrix valued function 
whose elements are $q$-th order polynomials in $\tau$, i.e., $(A(\tau))_{i,j}=\sum_{l=0}^qa_{i,j,l}\tau^l$ for $1\le i,j\le m$.

In this experiment, 
we deal with the cases where $q = 9$, $m = 10, 20$, and $T=[0,1]$, i.e., $T_{\rm min}=0$ and $T_{\rm max}=1$.
We generate 10 test problems for each of $m=10,20$ as follows:
We choose all entries of $A_0$ and the coefficients $a_{i,j,l}$ in $A(\tau)$ from the interval $[-1,1]$ randomly. 
Among those generated data sets, we use only data such that the semi-infinite constraint
includes at least one active constraint
at an optimum of \eqref{eig_semi}.
{Specifically, for each generated data, we} compute an optimum, say $\tilde{X}$, of the SDP obtained by removing the semi-infinite constraints.
If $\min_{1\le i\le 21}A\left(T_{\rm min} + \frac{(i-1)(T_{\rm max}-T_{\rm min})}{20}\right)\bullet \tilde{X}\le -10^{-3}$, which implies that $\tilde{X}$ does not satisfy the semi-infinite constraints, we adopt it as a valid data set.

We examine the performance of Algorithm~2
by comparing it
with the discretization method that uses SDPT3\,\cite{sdpt3} 
with the default setting to solve linear SDPs sequentially. 
The obtained results are shown in Tables~\ref{ta1} and \ref{ta2}, in which 
``ave.time(s)'' and ``$\Phi_0^{\ast}$'' stand for
the average running time in seconds and the average value of $\Phi_0$ at the solution output by the algorithm
``Disc." stands for the discretization method.
Moreover, ``AHO-like'', ``NT'', and ``H.K.M'' {stand for} Algorithm~2 combined with the scaling matrices $P=I,F(x^k)^{-\frac{1}{2}}$, and
$W^{-\frac{1}{2}}$, respectively. 

From the tables, we observe that 
computational time for ``AHO-like'' is largest among all. 
Actually, it spends around 3 seconds for $m = 10$ and 40 seconds for $m = 20$, while the others spend less than 1 second in all cases.  This is mainly due to
high computational costs for calculating the matrix $H_P$ defined by \eqref{eq:HP}, in which $\mathcal{L}_{F(x)}^{-1}$ must be dealt with. However, in the cases of ``NT'' and ``H.K.M'', $H_P$ can be handed more efficiently. 
Second, we observe that ``Disc.'' solves problems faster than Algorithm~2. 
This is  because an
SDP is solved very quickly with SDPT3 at each iteration of ``Disc.", and the number of SDPs solved is very small. In fact, only three or four SDPs are solved on average per run. However, we can see that our methods gain KKT points with higher accuracy than the discretization method.
More specifically, the values of $\Phi_0^{\ast}$ for
Algorithm~2 lie between $1.0\times 10^{-9}$ and $2.0\times 10^{-9}$, while 
those for the discretization method are around $10^{-6}$.
We also observed that 
Algorithm~2 skips Step~2.4 in most iterations, namely, $w^{k+1}$ is determined by the directions $\Delta_{\fr}w^{k+1}$ and $\Delta_{1}w^{k+1}$. Actually, Step~2.4 was skipped in more than 90\% of iterations.
Skipping Step~2.4 is desirable since the interior point SQP method performed in Step~2.4 is likely to solve multiple QPs and 
result in more computational cost than Steps~2.1 and 2.2.
Also, in most cases, the full step was accepted eventually and the value of $\Phi_{\mu_{k-1}}$ converged to 0 superlinearly.   
\begin{table}[h]
\centering
\small
\begin{minipage}{0.43\hsize}
\begin{tabular}{|c|c|c|}\hline
                  & ave.time(s) &  $\Phi_0^{\ast}$      \\     \hline\hline
AHO-like     &   2.63     & $1.39\cdot 10^{-9}$                          \\ \hline
NT              & 0.44& $1.39\cdot 10^{-9}$                       \\ \hline
H.K.M.         &   0.45     &  $1.39\cdot 10^{-9}$                     \\ \hline\hline
Disc.          &   0.54      &  $2.06\cdot 10^{-6}$           \\ \hline
\end{tabular}
\caption{Results for linear SISDPs with $m=10$}
\label{ta1}
\end{minipage}\hfill
\begin{minipage}{0.43\hsize}
\begin{tabular}{|c|c|c|}\hline
                  & ave.time(s)&   $\Phi_0^{\ast}$                       \\     \hline\hline
AHO-like     &   46.3      & $1.97\cdot10^{-9}$                \\ \hline 
NT              &     0.90      & $1.97\cdot10^{-9}$    \\ \hline
H.K.M.         &      0.90           &   $1.97\cdot10^{-9}$                      \\ \hline\hline
Disc.          &       0.40           &   $1.34\cdot 10^{-6}$                                 \\ \hline
\end{tabular}
\caption{Results for linear SISDPs with $m=20$}
\label{ta2}
\end{minipage}
\end{table}

\subsection{Nonlinear SISDPs}
Next, we solve the following SISDP whose objective function is nonlinear: 
\begin{align}
\begin{array}{rcl}
\displaystyle{\mathop{\rm Minimize}_{x\in \R^{\frac{m(m+1)}{2}}}}& &\frac{1}{2}x^{\top}Mx+c^{\top}x+\omega{\|x\|^4}\\
\mbox{subject to}& & 
                     \sum_{i=1}^{n}\tau^{i-1}x_i\le \sum_{i=1}^n\tau^{2i} + \sin(9\pi \tau)+2  \ \ (\tau\in T)\\
                         & & X + \kappa I\in S^m_+
\end{array}\label{eig_semi2}
\end{align}
with $\omega>0$, $\kappa >0$, and $n:=m(m+1)/2$.
The objective function is 
not convex in general but coercive in the sense that 
$f(x)\to \infty$ as $\|x\|\to \infty$, and thus the considered problem is guaranteed to have at least one global optimum.
We deal with the cases of $m = 10,20$.
{For each of $m=10,20$,} all the elements of $M\in S^m$ and $c\in \R^n$ are randomly generated from the interval $[-1,1]$.
We set $T=[0,1]$ and $\kappa = \omega=0.01$. In Step~2 of the discretization method, we use
the primal-dual interior point method
\cite{yabe} to solve finitely relaxed SISDPs.

We show the results in Tables~\ref{ta3} and \ref{ta4}, 
where each column and row has the same meaning as in Tables~\ref{ta1} and \ref{ta2}. 
From the tables, ``AHO-like'' spends the largest CPU-time like in linear SISDPs.
We observe that Algorithm~2 (AHO-like, NT, H.K.M.) successfully obtains KKT points with higher accuracy than the discretization method. Actually, the values of $\Phi_0^{\ast}$
obtained by Algorithm~2 lie between $10^{-9}$ and $2\times 10^{-9}$, while those for the discretization method are around $10^{-6}$.
Compared with the case of linear SISDPs, we observed that the rate of skipping Step~2-4 was less. 
Actually, Step~2-4 was used at about 15\% of iterations when $m=10$
and about 24\% when $m=20$,
while it was used only in a few early iterations for linear SISDPs.
This might be caused by the nonlinearity of the objective function.
\begin{table}[h]
\centering
\small
\begin{minipage}{0.43\hsize}
\begin{tabular}{|c|c|c|}\hline
                  & ave.time(s) &  $\Phi_0^{\ast}$                       \\     \hline\hline
AHO-like     &  3.16        & $1.39\cdot 10^{-9}$               \\ \hline
NT              &   0.86       & $1.39\cdot 10^{-9}$               \\ \hline
H.K.M.         &   0.85        &  $1.39\cdot 10^{-9}$           \\ \hline\hline
Disc.          &    1.27        &  $9.62\cdot 10^{-7}$                \\ \hline
\end{tabular}
\caption{Results for the nonlinear SISDP with $m=10$}
\label{ta3}
\end{minipage}\hfill
\begin{minipage}{0.43\hsize}
\begin{tabular}{|c|c|c|}\hline
                  & ave.time(s)&   $\Phi_0^{\ast}$                        \\     \hline\hline
AHO-like     &   50.3             & $1.97\cdot10^{-9}$              \\ \hline 
NT              &     4.06           & $2.32\cdot10^{-9}$  \\ \hline
H.K.M.         &      4.00           &   $2.32\cdot10^{-9}$                       \\ \hline\hline
Disc.          &       8.08           &   $8.06\cdot 10^{-7}$                                  \\ \hline
\end{tabular}
\caption{Results for the nonlinear SISDP with $m=20$}
\label{ta4}
\end{minipage}
\end{table}

%% file: appendix.tex
\section*{Appendix}
\def\thesection{A}
In the appendix, we prove Proposions\,\ref{prop:0511}, \ref{lem:1203-2}, and Lemma\,\ref{prop:0604-1}.
We begin with
giving some lemmas that help to show Proposion\,\ref{prop:0511}.
\begin{lemma}\label{lem:0424}
Let $X\in S^m_+$, $Y\in S^m$ and $\mu\ge 0$.
Then, 
\begin{enumerate}
\item 
$
\|XY-YX\|_F\le 2\|X\circ Y-\mu I\|_F
$ and
\item 
$
\|\mathcal{L}_X\mathcal{L}_Y-\mathcal{L}_Y\mathcal{L}_X\|_2\le \|X\circ Y-\mu I\|_F.
$
\end{enumerate}
\end{lemma}
\begin{proof}
Using some orthogonal matrix $\mathcal{O}\in \R^{m\times m}$, we make an eigenvalue decomposition of $X$: $\mathcal{O}^{\top}X\mathcal{O}=D$ with $D\in \R^{m\times m}$ being a diagonal matrix. Denote the $i$-th diagonal entry of $D$
by $d_i\ge 0$ for $i=1,2,\ldots,m$. 
Let $\tilde{Y}:=\mathcal{O}^{\top}Y\mathcal{O}$ with the $(i,j)$-th entry $\tilde{y}_{ij}$ for $1\le i,j\le m$.
\begin{enumerate}
\item We have the desired result from 
\begin{align}
\|XY-YX\|_F^2&=\|\mathcal{O}^{\top}X\mathcal{O}\mathcal{O}^{\top}Y\mathcal{O}-\mathcal{O}^{\top}Y\mathcal{O}\mathcal{O}^{\top}X\mathcal{O}\|_F^2\notag\\
              &=\|D\tilde{Y}-\tilde{Y}D\|_F^2\notag\\
              &=\sum_{1\le i,j\le m}(d_{i}-d_{j})^2\tilde{y}_{ij}^2\notag\\
              &\le \sum_{1\le i\neq j\le m}(d_{i}+d_{j})^2\tilde{y}_{ij}^2\notag\\
              &\le\sum_{1\le i\neq j\le m}(d_{i}+d_{j})^2\tilde{y}_{ij}^2+\sum_{i=1}^m(2d_{i}\tilde{y}_{ii}-2\mu)^2\notag\\
              &=\|D\tilde{Y}+\tilde{Y}D-2\mu I\|_F^2\notag\\
              &=\|XY+YX-2\mu I\|_F^2\notag \\
              &=4\|X\circ Y-\mu I\|_F^2,\notag
\end{align}
where the first inequality follows from $d_{i}\ge 0$ for $i=1,2,\ldots,m$.
\item
{By direct calculation, we have}
\begin{align}
\|\mathcal{L}_X\mathcal{L}_Y-\mathcal{L}_Y\mathcal{L}_X\|_2
&=\max_{\|Z\|_F=1}\|\mathcal{L}_X\mathcal{L}_YZ-\mathcal{L}_Y\mathcal{L}_XZ\|_F\notag \\
&=\max_{\|Z\|_F=1}\frac{\|(XY-YX)Z-Z(XY-YX)\|_F}{4}\notag \\
&\le  \frac{\|XY-YX\|_F}{2}\notag\\
&\le \|X\circ Y-\mu I\|_F,\notag                                                                               
\end{align}
where the {second} inequality follows from item~1.
\end{enumerate}
\hspace{\fill}$\square$
\end{proof}

\begin{lemma}\label{lem:0424-2}
{Let $(X_{\ast},Y_{\ast})\in S^m_+\times S^m_+$ satisfy the strict complementarity condition that $X_{\ast}{\circ}Y_{\ast}=O$ and $X_{\ast}+Y_{\ast}\in S^m_{++}$. 
Let $\{\mu_r\}\subseteq \R_{++}$ and $\{(X_r,Y_r)\}\subseteq S^m_{++}\times S^m_{++}$ be sequences  such that $\lim_{r\to\infty}\mu_r=0$ and $\lim_{r\to\infty}(X_r,Y_r)=(X_{\ast},Y_{\ast})$.}
Let spectral decompositions of $X_{\ast}$ and $Y_{\ast}$ be 
$$
\mathcal{O}_{\ast}^{\top}X_{\ast}\mathcal{O}_{\ast}=\begin{pmatrix}
D_{X_{\ast}}&O\\
O&O
\end{pmatrix},\  
\mathcal{O}_{\ast}^{\top}Y_{\ast}\mathcal{O}_{\ast}
=
\begin{pmatrix}
O&O\\
O&D_{Y_{\ast}}
\end{pmatrix}
$$
using some orthogonal matrix $\mathcal{O}_{\ast}\in \R^{m\times m}$
and positive diagonal matrices $D_{X_{\ast}}\in S^p_{++}$ and $D_{Y_{\ast}}\in S^q_{++}$ with $p+q=m$.
Furthermore, suppose $p,q>0$ and 
choose a sequence of orthogonal matrices $\{\mathcal{O}_{r}\}\subseteq  \R^{m\times m}$ such that
\begin{equation*}
\mathcal{O}_r^{\top}X_r\mathcal{O}_r=\begin{pmatrix}
D_{X_{r}}&O\\
O&E_{X_{r}}
\end{pmatrix},\ \lim_{r\to\infty}\mathcal{O}_r=\mathcal{O}_{\ast}
\end{equation*}
with $D_{X_r}\in \R^{p\times p}$ and $E_{X_r}\in \R^{q\times q}$ being positive diagonal matrices for $r\ge 1$. 
(Notice that $\lim_{r\to \infty}E_{X_r}$=O.)
If $\|X_r\circ Y_r-\mu_r I\|=o(\mu_r)$, then 
\begin{equation}
\lim_{r\to\infty}\frac{1}{\mu_r}E_{X_r}=D_{Y_{\ast}}^{-1}.\label{eq:0424-1}
\end{equation} 
\end{lemma}
\begin{proof}
Let $\tilde{Y}_r:=\mathcal{O}_r^{\top}Y_r\mathcal{O}_r$ and 
$\tilde{y}^r_{ii}$ and 
$e^r_i$ be the $i$-th diagonal entry of $\tilde{Y}_r$ and $E_{X_r}$, respectively for any $i=p+1,p+2,\ldots,m$.
Since $\|X_r\circ Y_r-\mu_rI\|_F=o(\mu_r)$ and 
\begin{align*}
\|X_r\circ Y_r-\mu_{r}I\|_F&=\left\|\begin{pmatrix}
D_{X_{r}}&O\\
O&E_{X_{r}}
\end{pmatrix}\circ\tilde{Y}_r-\mu_r I\right\|_F\notag \\
&\ge \sqrt{
\sum_{i=p+1}^{m}(e^r_{i}\tilde{y}_{ii}^r-\mu_r)^2},                                   
\end{align*}
we have
\begin{equation}
0=\lim_{r\to \infty}\frac{\sqrt{\sum_{i=p+1}^{m}\left(e^r_i\tilde{y}^r_{ii}-\mu_r\right)^2}}{\mu_r}
=\lim_{r\to \infty}\sqrt{\sum_{i=p+1}^{m}\left(\frac{e^r_i}{\mu_r}\tilde{y}^r_{ii}-1\right)^2},
\notag
\end{equation}
which yields $\lim_{r\to\infty}\frac{e^r_i}{\mu_r}\tilde{y}^r_{ii}=1$ for any $i=p+1,\ldots,m$.
Notice that, 
for $i\ge p+1$,
$\{\tilde{y}^r_{ii}\}$ converges to
the $i$-th positive diagonal entry of $D_{Y_{\ast}}$.
In view of these facts, we obtain \eqref{eq:0424-1}. 
\hspace{\fill}$\square$
\end{proof}
\subsubsection*{Proof of Proposition\,\ref{prop:0511}}
For the case where $X_{\ast}\in S^m_{++}$, it is easy to prove the desired result.
So, we consider the case of $X_{\ast}\in S^m_+\setminus S^m_{++}$.
Let $\lambda_r>0$ be the smallest eigenvalue of $X_r$. 
Notice that $\lambda_r\to 0\ (r\to \infty)$ and, by Lemma~\ref{lem:0424-2}, $\lim_{r\to\infty}\frac{\lambda_r}{\mu_r}$ exists and is positive.
Thus, we also have 
\begin{equation}
\lim_{r\to\infty}\frac{\mu_r}{\lambda_r}>0.\label{eq:0604}
\end{equation} 
Note that, for any $X\in S^m$ having $m$ eigenvalues $\alpha_1\le \alpha_2\le\cdots \le \alpha_m$,
the corresponding symmetric linear operator $\mathcal{L}_X$ has $m(m+1)/2$ eigenvalues
$
\alpha_1,\alpha_2,\ldots,\alpha_m,\{(\alpha_i+\alpha_j)/2\}_{i\neq j}.
$ 
This fact yields that the maximum eigenvalue of the operator $\mathcal{L}_{X_r}^{-1}$ is $\lambda_r^{-1}$. 
Therefore, we have $\|\mathcal{L}_{X_r}^{-1}\|_2=\lambda_r^{-1}$ for any $r\ge 0$.
It then follows that 
\begin{align}
\|
\mathcal{L}_{X_r}\mathcal{L}_{Y_r}\mathcal{L}_{X_r}^{-1}-\mathcal{L}_{Y_r}
\|_2
&\le
\|\mathcal{L}_{Y_r}\mathcal{L}_{X_r}-\mathcal{L}_{X_r}\mathcal{L}_{Y_r}
\|_2\|\mathcal{L}_{X_r}^{-1}\|_2\notag \\
&\le 
\mu_r\|\mathcal{L}_{X_r}^{-1}\|_2
\frac{\|
X_r\circ Y_r-\mu_rI
\|_F}{\mu_r}   \notag\\
&=\frac{\mu_r}{\lambda_r}\frac{\|
X_r\circ Y_r-\mu_rI
\|_F}{\mu_r},
\end{align}
where the second inequality follows from Lemma\,\ref{lem:0424}.
This relation together with \eqref{eq:0604} and $\|X_r\circ Y_r-\mu_rI\|_F=O(\mu_r^{1+{\zeta}})$
implies 
$
\|
\mathcal{L}_{X_r}\mathcal{L}_{Y_r}\mathcal{L}_{X_r}^{-1}-\mathcal{L}_{Y_r}
\|_2=O(\mu_r^{\zeta}).$
\\
\hspace{\fill}$\square$
\subsubsection*{Proof of Proposition\,\ref{lem:1203-2}}
Define $\Phi_r(s):=\left(X_r+s\Delta X_r\right)\circ \left(Y_r+s\Delta Y_r\right)$ for $s\in [0,1]$ and each $r$. 
By using the fact that $\|X\|_F\ge |\lambda_{\min}(X)|$ for any $X\in S^m$, the conditions\,\eqref{al:1202-1}--\eqref{al:1202-3} yield that there exists some $\theta>0$ such that
\begin{align}
&\lambda_{\min}\left(\Delta X_r\circ \Delta Y_r\right)\ge -\theta\mu_r^2,\label{al:1202-4}\\
&\lambda_{\min}\left(X_r\circ Y_r\right)\ge \mu_r-\theta \mu_r^{1+\zeta},\label{al:1202-5} \\
&\lambda_{\min}\left(Z_r-\hmu_rI\right)\ge -\theta \hmu_r^{1+\hzeta}.\label{al:1202-6}
\end{align}
Then,  it holds that 
\begin{align}
\lambda_{\min}(\Phi_r(s))&=\lambda_{\min}\left(
X_r\circ Y_r+s X_r\circ \Delta Y_r+s Y_r\circ \Delta X_r+s^2\Delta X_r\circ \Delta Y_r
\right)\notag\\
           &=\lambda_{\min}\left(
(1-s)X_r\circ Y_r+s(Z_r-\hmu_r I)+s\hmu_r I+s^2\Delta X_r\circ \Delta Y_r\right)\notag\\
&\ge (1-s)\lambda_{\min}\left(X_r\circ Y_r\right)+s\lambda_{\min}(Z_r-\hmu_r I)\notag\\
&\hspace{5em}+s\lambda_{\min}(\hmu_r I)+s^2\lambda_{\min}\left(\Delta X_r\circ \Delta Y_r\right)\notag\\ 
            &\ge (1-s)\left(\mu_r-\theta \mu_r^{1+\zeta}\right) 
-s\theta\hmu_r^{1+\hzeta}+s\hmu_r-s^2 \theta\mu_r^2\notag\\
&=:\phi_r(s) \notag
\end{align}
for any $r$ sufficiently large and $s\in [0,1]$, where the first inequality follows from the fact that $\lambda_{\min}(A+B)\ge \lambda_{\min}(A)+\lambda_{\min}(B)$ for $A, B\in S^m$ and the second inequality is due to \eqref{al:1202-4}--\eqref{al:1202-6} and $s\in [0,1]$.
Notice that $\phi_r(s)$ is concave and quadratic.
Then, for any $r$ sufficiently large, we have
$\phi_r(s)>0\ (s\in [0,1])$ since $0<\zeta,\hzeta<1$, $\lim_{r\to \infty}(\mu_r,\hmu_r)=(0,0)$, and \eqref{al:1202-7} imply that 
$\phi_r(0)=\mu_r-\theta \mu_r^{1+\zeta}>0$ and 
$\phi_r(1)=\hmu_r-\theta\hmu_r^{1+\hzeta}-\theta\mu_r^2>0$ for sufficiently large $r$.
This means that $\lambda_{\min}(\Phi_r(s))\ge \phi_r(s)>0\ (s\in [0,1])$ and therefore 
\begin{equation}
\Phi_r(s)\in S^m_{++}\ (s\in [0,1]), \label{eq:Phi}
\end{equation}
from which we can derive $X_r+\Delta X_r\in S^m_{++}$ and $Y_r+\Delta Y_r\in S^m_{++}$.
Actually, for contradiction, suppose that either one of these two conditions is not true.
We can assume $X_r+\Delta X_r\notin S^m_{++}$ without loss of generality. 
Recall that $X_r\in S^m_{++}$.
Then, there exists some $\bar{s}\in (0,1]$ such that $X_r+\bar{s}\Delta X_r\in S^m_{+}\setminus S^m_{++}$.
Therefore, we can find some nonzero vector $d\in \R^n$ such that $(X_r+\bar{s}\Delta X_r)d=0$.
From this fact, we readily have
\begin{align}
d^{\top}\Phi_r(\bar{s})d&=\frac{
d^{\top}(X_r+\bar{s}\Delta X_r)(Y_r+\bar{s}\Delta Y_r)d+
d^{\top}(Y_r+\bar{s}\Delta Y_r)(X_r+\bar{s}\Delta X_r)d
}{2}=0,\notag
\end{align}
which contradicts \eqref{eq:Phi}. Hence, we conclude that $X_r+\Delta X_r\in S^m_{++}$ and $Y_r+\Delta Y_r\in S^m_{++}$ for all $r$ sufficiently large. The proof is complete. \hspace{\fill}$\square$
\subsubsection*{Proof of Lemma\,\ref{prop:0604-1}}
To begin with, by $w^k\in \mathcal{N}_{\mu_{k-1}}^{\epsilon_{k-1}}$ and $\epsilon_{k-1}=\gamma_1\mu_{k-1}^{1+\alpha}$, it follows that
\begin{align}
&\left\|\nabla f(x^k)+\sum_{i=1}^{p(x^{\ast})}\nabla \hat{g}_i(x^k)y_i^k-(F_i\bullet V_k)_{i=1}^n\right\|=o(\tbkpre),\ \|F(x^k)\circ V_k\|_F=\Theta(\tbkpre),\label{al:1012-1}\\
&\left|\sum_{i=1}^{p(x^{\ast})}y_i^k\hat{g}_i(x^k)\right|=o(\tbkpre),\ \max_{1\le i\le p(x^{\ast})}(\hat{g}_i(x^k))_+=o(\tbkpre)\label{al:1012}
\end{align}
together with $y_i^k\ge 0\ (i=1,2,\ldots,p(x^{\ast}))$.
Then, \eqref{al:1012} implies $|\hat{g}_i(x^k)|=o(\tbkpre)\ (i=1,2,\ldots,p(x^{\ast}))$, which together 
with \eqref{al:1012-1} and \eqref{al:1012} yields $\|\Phi_0(\tilde{w}^k)\|=\Theta(\tbkpre)$.
We then have $\tbkpre=\Theta(\|\Phi_{0}(\tilde{w}^k)\|)$.

We next prove $\tbkpre=\Theta(\|w^k-w^{\ast}\|)$. 
Notice that by Assumption~B-\ref{sc}, 
for sufficiently large $k$, $y^k_i>0\ (i\in I_a(x^{\ast}))$ and $y^k_i=0\ (i\in \{1,2,\ldots,p(x^{\ast})\}\setminus I_a(x^{\ast}))$, which together with $y_i^{\ast}=0\ (i\in \{1,2,\ldots,p(x^{\ast})\}\setminus I_a(x^{\ast}))$ implies $\|\tilde{w}^k-\tilde{w}^{\ast}\|=\|{w}^k-{w}^{\ast}\|$.
Thus, to show the desired result, we have only to prove $\|\Phi_{0}({w}^k)\|=\Theta(\|\tilde{w}^k-\tilde{w}^{\ast}\|)$.
In other words, it suffices to show that 
the sequence of positive numbers $\{\zeta_k\}$ is bounded above and away from zero, where $\zeta_k:={\|\Phi_{0}(\tilde{w}^k)\|}/{\|\tilde{w}^k-\tilde{w}^{\ast}\|}$. Note that 
\begin{equation*}
\zeta_k=\frac{\|\Phi_{0}(\tilde{w}^k)-\Phi_{0}(\tilde{w}^{\ast})\|}{{\|\tilde{w}^k-\tilde{w}^{\ast}\|}}=
\left\|\mathcal{J}\Phi_{0}(\tilde{w}^{\ast})\frac{\tilde{w}^k-\tilde{w}^{\ast}}{\|\tilde{w}^k-\tilde{w}^{\ast}\|}+\frac{O(\|\tilde{w}^k-\tilde{w}^{\ast}\|^2)}{\|\tilde{w}^k-\tilde{w}^{\ast}\|}\right\|.
\end{equation*}
Obviously, $\zeta_k$ is bounded from above. 
To show $\zeta_k$ is bounded away from zero, suppose to the contrary.
Then, without loss of generality, we can assume that 
$\lim_{k\to \infty }\zeta_k=0$, 
and hence there exists some $d^{\ast}$ with $\|d^{\ast}\|=1$ such that $\lim_{k\to\infty}\frac{\tilde{w}^k-\tilde{w}^{\ast}}{\|\tilde{w}^k-\tilde{w}^{\ast}\|}=d^{\ast}$ and $\mathcal{J}\Phi_0(\tilde{w}^{\ast})d^{\ast}=0$.
However, this contradicts
the nonsingularity of $\mathcal{J}\Phi_{0}(\tilde{w}^{\ast})$ from Assumption~C-\ref{nonsing}.
We have the desired conclusion.
\hspace{\fill}$\square$